\documentclass[12pt]{amsart}

\usepackage[margin=2.54cm]{geometry}

\usepackage{amssymb}
\usepackage{amsmath}
\usepackage{amsfonts}
\usepackage{mathrsfs,enumitem}
\usepackage[utf8]{inputenc}
\usepackage{color}
\usepackage{ulem} 

\newcounter{tmpc} 

\newcommand{\disk}{\ensuremath{\mathbb{D}} } 
\newcommand{\sphere}{\overline{\Bbb{C}}} 
\newcommand{\riem}{\Sigma}  
\renewcommand{\Bbb}[1]{\ensuremath{\mathbb{#1}}}

\newcommand{\Oqc}{\mathcal{O}^{\text{qc}}} 

\newcommand{\Mon}{\widetilde{\mathcal{M}}(0,n+1)}
\newcommand{\Mgn}{\widetilde{\mathcal{M}}(g,n+1)}

\newcommand{\s}{$s$\nobreakdash} 
\newcommand{\n}{$n$\nobreakdash}
\newcommand{\ith}{$i$\nobreakdash-th }
\newcommand{\1}{$1$\nobreakdash}

\newcommand{\Oqco}{\mathcal{O}^{\mathrm{WP}}\!} 

\newcommand{\gr}{\operatorname{\mathbf{G}}}
\newcommand{\Rig}{\mathcal{R}}

\theoremstyle{plain}
        \newtheorem{theorem}{Theorem}[section]
        \newtheorem{lemma}[theorem]{Lemma}
        \newtheorem{proposition}[theorem]{Proposition}
        \newtheorem{corollary}[theorem]{Corollary}

\theoremstyle{definition}
        \newtheorem{definition}[theorem]{Definition}

\theoremstyle{remark}
    \newtheorem{remark}[theorem]{Remark}

\numberwithin{equation}{section} 
\numberwithin{figure}{section} 

\author{David Radnell}
\address{David Radnell \\ Department of Mathematics and Statistics \\
American University of Sharjah \\
PO Box 26666, Sharjah \\ United Arab Emirates} \email{dradnell@aus.edu}

\author{Eric Schippers}
\address{Eric Schippers \\ Department of Mathematics \\
University of Manitoba\\
Winnipeg, Manitoba \\  R3T 2N2 \\ Canada}
\email{eric\_schippers@umanitoba.ca}

\author{Wolfgang Staubach}
\address{Wolfgang Staubach\\ Department of Mathematics\\
Uppsala University\\
Box 480\\ 751 06 Uppsala\\ Sweden}
\email{wulf@math.uu.se}



\title[Dirichlet space of multiply connected domains]{Dirichlet space of multiply connected domains with Weil-Petersson class boundaries}
\begin{document}

\begin{abstract}
 The restricted class of quasicircles sometimes called the ``Weil-Petersson-class'' has been a subject of
 interest in the last decade.  In this paper we establish a Sokhotski-Plemelj jump formula
 for WP-class quasicircles, for boundary data in a certain conformally invariant Besov space.
 We show that this Besov space is precisely the set of traces on the boundary of harmonic functions of finite
 Dirichlet energy on the WP-class quasidisk.

 We apply this result to
 multiply connected domains $\riem$ which are the complement of $n+1$
 WP-class quasidisks.  Namely, we give a bounded isomorphism between the Dirichlet space $\mathcal{D}(\riem)$ of $\riem$ and a
 direct sum of Dirichlet spaces, $\mathcal{D}^-$, of the unit disk.
  Writing the quasidisks as images of the disk under conformal maps
   $(f_0,\ldots,f_n)$, we also show that $\{ (h \circ f_0,\ldots,h \circ f_n) \,: h \in \mathcal{D}(\riem)\}$
 is the graph of a certain bounded Grunsky operator on $\mathcal{D}^-$.
\end{abstract}

\maketitle

\begin{section}{Introduction}
\begin{subsection}{Statement of results and preliminaries} \label{se:results}
In this section we state the necessary definitions and outline the results.  Section
\ref{se:literature} contains a brief discussion of the literature from an analytic
point of view.
Applications and further literature can be
found in Section \ref{se:applications}.

 Let $\sphere$ denote the Riemann sphere, $\disk^+ = \{z : |z|<1 \}$ and
 $\disk^- = \{ z : |z|>1 \} \cup \{\infty\}$.   Denote the circle $|z|=1$ by
 $\mathbb{S}^1$.  Consider an $n+1$-tuple of
 maps $(f_0,\ldots,f_n)$ with the following properties.
 \begin{enumerate}[label=(\alph*)]
  \item $f_i \colon \disk^+ \rightarrow D_i$ are bijective holomorphic
  maps onto open sets $D_i \subseteq \mathbb{C}$, for $i=0,\ldots,n-1$, and $f_n \colon \disk^-
  \rightarrow \sphere$ is a bijective holomorphic map onto $D_n \subseteq \sphere$
   taking $\infty$ to $\infty \in D_n$.
  \item The closures of the images are non-intersecting; that is $\overline{D_i} \cap \overline{D_j}$ is
  empty for all $i \neq j$.
  \item The boundaries $\partial D_i$ are WP-class quasicircles for $i=1,\ldots,n$.
 \end{enumerate}
 For the definition of WP-class quasicircles, see Section \ref{se:rigged_moduli_space_def} ahead.
 Let $\riem = \sphere \backslash \left( \cup_{i=1}^n \overline{D_i} \right)$ denote
 the interior of the complement of the images of the maps.  Denote the $i$th
 boundary curve by $\partial_i \riem = \partial D_i$.  We are concerned with the following problem, which is motivated by conformal field theory and quasiconformal Teichm\"uller theory.

Let $\mathcal{H}(\mathbb{S}^1)$ be the set of functions on the circle $\mathbb{S}^1$ with half-order derivative
in $L^2$.
\begin{description}
 \item[Problem]  Let
 \[  \mathcal{D}(\riem) = \left\{ h \colon \riem \rightarrow \mathbb{C} \, :\, h \text{ holomorphic and }   \iint_{\riem} |h'|^2 <\infty \right\}  \]
 denote the Dirichlet space of $\riem$.
  Consider the map from  $\mathcal{D}(\riem)$ to $\mathcal{H}(\mathbb{S}^1) \oplus \cdots \oplus \mathcal{H}(\mathbb{S}^1)$
  defined by
   \begin{equation} \label{eq:composition_shorthand}
    h  \longmapsto  \left(\left. h \circ f_0 \right|_{\mathbb{S}^1},\ldots, \left. h \circ f_n \right|_{\mathbb{S}^1}\right).
   \end{equation}
 Characterize the image of this map
 analytically and algebraically.
\end{description}

 Of course, one sees immediately that there are several analytic problems
 which need to be resolved in order to make this problem meaningful.
 In this paper, we formulate a natural analytic setting for this problem, and show that the image is the graph of an operator related to the Grunsky matrix.  To do so, we extend the Sokhotski-Plemelj
 jump decomposition to a certain Besov space $\mathcal{H}(\Gamma)$ on any WP-class quasicircle $\Gamma$ (see Definition \ref{defn: besov space norm}
 and equation (\ref{eq:Besov_space_definition}) ahead).  This Besov space is naturally related to
 the Dirichlet space and is conformally invariant in a certain sense. In the case that $\Gamma= \mathbb{S}^1$, it reduces
 to the space of function with square-integrable half-order derivatives.

 We now make these statements precise.
  Let $\Omega$ be a simply-connected domain in $\sphere$
such that $\infty \notin \partial \Omega$.
The Dirichlet space is
\begin{equation} \label{eq:Dirichlet_hol_in_def}
  \mathcal{D}(\Omega) = \left\{ h \colon \Omega \rightarrow \mathbb{C} \,:\,
    h \text{ holomorphic and }  \iint_{\Omega} |h'|^2 \,dA <\infty \right\},
\end{equation}
if $\infty \notin \Omega$, and if $\infty \in \Omega$, then
\begin{equation} \label{eq:Dirichlet_hol_out_def}
 \mathcal{D}(\Omega) = \left\{ h \colon \Omega \rightarrow \mathbb{C} \,:\,
    h \text{ holomorphic, }  h(
    \infty)= 0 \text{ and }  \iint_{\Omega} |h'|^2 \,dA <\infty \right\}.
\end{equation}
We endow $\mathcal{D}(\Omega)$ with the norm
\[  \| h\| = \left( \iint_\Omega |h'|^2 \,dA \right)^{1/2} \]
if $\infty \in \Omega$, and with the norm
\[  \Vert h\Vert = \left( |h(0)|^2 + \iint_{\Omega} |h'|^2 \,dA \right)^{1/2}  \]
if $0 \in \Omega$. (The convention that functions in $\mathcal{D}(\Omega)$
satisfy $h(\infty)=0$ if $\infty$ is in $\Omega$
 is a matter of convenience, which we adopt because functions obtained by a Cauchy integral will have this property.)

The solution of the problem requires the solution of the following analytic results which are of independent interest.

\medskip

\noindent \textbf{Analytic results.}

\begin{enumerate}
\item Let $\Gamma$ be a WP-class quasicircle bounding a WP-class
quasidisk $\Omega$.  A function $h$ on $\Gamma$ is in $\mathcal{H}(\Gamma)$ if and only if $h$ is the trace of a
complex-valued harmonic function $H$ on $\Omega$ with finite Dirichlet energy.  The
restriction and extension operators are bounded.  (Theorem \ref{th:trace_of_Dirichlet}).
\smallskip
\item For bounded WP-class quasidisks $\Omega_1$ and a conformal map $f \colon \Omega_1 \rightarrow \Omega_2$, the composition operator
\begin{align*}
   C_f \colon \mathcal{H}(\partial \Omega_2) & \longrightarrow  \mathcal{H}(\partial \Omega_1)  \\
   h & \longmapsto  h \circ f
\end{align*}
is a bounded linear isomorphism.  (Theorem \ref{th:one-sided_composition}).
\smallskip
\item  For any WP-class quasicircle $\Gamma$ (in particular, for $\Gamma=\partial_i \riem$), there is
   a direct sum decomposition $\mathcal{H}(\Gamma) = \mathcal{D}(\Omega^+) \oplus \mathcal{D}(\Omega^-)$
   where $\Omega^+$ and $\Omega^-$ are the bounded and unbounded components of $\sphere \backslash \Gamma$
   respectively.  The Cauchy projections onto $\mathcal{D}(\Omega^\pm)$ are
   bounded.  (Theorem \ref{th:Cauchy_bounded}).
   \setcounter{tmpc}{\theenumi}
\end{enumerate}

  The decomposition in item (3) is of course given by the Sokhotski-Plemelj jump formula for the
  Cauchy kernel.  Note that result (2) demonstrates that the Besov space $\mathcal{H}(\Gamma)$ is in
  some sense conformally invariant.

  To continue stating the results, we need a bit more notation.  Denote
  \begin{equation} \label{eq:H_many_definition}
   \mathcal{H} =  \mathcal{H}(\mathbb{S}^1) \oplus \cdots \oplus \mathcal{H}(\mathbb{S}^1)
  \end{equation}
  where there are $n+1$ terms in the above sum, one corresponding to each boundary curve.
  and
  \begin{align} \label{eq:Dpm_definition}
   \mathcal{D}^+ & = \mathcal{D}(\disk^+) \oplus \mathcal{D}(\disk^+) \oplus \cdots \oplus \mathcal{D}(\disk^-) \\
   \mathcal{D}^- &= \mathcal{D}(\disk^-) \oplus \mathcal{D}(\disk^-) \oplus \cdots \oplus \mathcal{D}(\disk^+) \nonumber
  \end{align}
   where in both sums all the copies except for the last are identical. The direct sum
   decomposition $\mathcal{H}(\mathbb{S}^1) = \mathcal{D}(\disk^+) \oplus \mathcal{D}(\disk^-)$ \cite{NS} yields the decomposition
  \[  \mathcal{H} = \mathcal{D}^- \oplus \mathcal{D}^+  \]
  and bounded projection operators $\mathcal{P}_\pm \colon \mathcal{H} \rightarrow \mathcal{D}^\pm$.
  Similarly, denote
  \begin{equation}  \label{eq:Hriem_definition}
    \mathcal{H}(\partial \riem) = \mathcal{H}(\partial_0 \riem) \oplus \cdots \oplus \mathcal{H}(\partial_n \riem).
  \end{equation}
  For $f=(f_0,\ldots,f_n)$ satisfying (a), (b) and (c), let
  \begin{align}  \label{eq:C(f)_definition}
   \mathcal{C}_f \colon \mathcal{H}(\partial \riem) & \longrightarrow \mathcal{H} \\ \nonumber
   (h_0,\ldots,h_n) & \longmapsto (h_0 \circ f, \ldots, h_n \circ f)
  \end{align}
  and
  \begin{align}  \label{eq:I(f)_definition}
  \mathbf{I}_f \colon \mathcal{D}^-
    &\longrightarrow \mathcal{D}(\Sigma) \\ \nonumber
   g = (g_0,\ldots,g_n) & \longmapsto  \mathcal{P}(\riem)\,\mathcal{C}_{f^{-1}} (g_0,\ldots,g_n)
  \end{align}
  where $\mathcal{P}(\riem)$ is defined by
  \begin{align} \label{eq:Priem_definition}
   \mathcal{P}(\riem) \colon \mathcal{H}(\partial \riem) & \longrightarrow \mathcal{D}(\riem) \\ \nonumber
   (h_0, \ldots, h_n) & \longmapsto \sum_{i=0}^{n} \frac{1}{2\pi i} \int_{\partial_i \riem}
    \frac{h_i(\zeta)}{\zeta - z} d\zeta
  \end{align}
  and it is understood that the output is a function of $z \in \riem$.
  Finally define $\mathcal{W}_f: \mathcal{D}^- \rightarrow \mathcal{H}$ by
  \begin{equation} \label{eq:Wf_definition}
    \mathcal{W}_f = \mathcal{C}_f \, \mathbf{I}_f.
  \end{equation}
 Note that we have suppressed a trace operator taking elements of $\mathcal{D}(\riem)$ to their boundary values in $\mathcal{H}(\partial \riem)$.

   \medskip

   \noindent \textbf{Geometric results.}

\begin{enumerate}
\setcounter{enumi}{\thetmpc}
    \item  $\mathbf{I}_f \colon \mathcal{D}^- \rightarrow \mathcal{D}(\riem)$ is a bounded linear isomorphism.
    (Theorem \ref{th:I_multi_isomorphism}).
    \smallskip
  \item  $\mathcal{P}_- \, \mathcal{W}_f$ is the identity and  $\mathcal{P}_+ \,
   \mathcal{W}_f$ is a bounded linear isomorphism.  In the standard basis $\{z^n\}$, the blocks of
   $\mathcal{P}_+\, \mathcal{W}_f$  are Grunsky matrices on the diagonal,
   and so-called generalized Grunsky matrices off the diagonal.  (Theorem \ref{th:two_projections_id_Grunsky},
   equations (\ref{eq:Faber_phin}) and (\ref{eq:relation_to_regular_Grunsky})).
   \smallskip
 \item $\mathcal{C}_f \, \mathcal{D}(\riem)$ is the
 graph of $\mathcal{P}_- \mathcal{W}_f$ in the fixed space $\mathcal{H}$.  (Corollary \ref{co:main_question_answered}).
 \end{enumerate}
  Result (4) relates to an isomorphism of Shen \cite{ShenFaber} (see Remark \ref{re:Shen_comparison}
  ahead).  In the notation of this section, the image of the map (\ref{eq:composition_shorthand}) is
  just $\mathcal{C}_f \mathcal{D}(\riem)$.  Thus
  results (5) and (6) answer the main problem and show that $\mathcal{C}_{f} \, \mathcal{D}(\riem)$
   is the graph of a kind of generalized Grunsky operator.
\end{subsection}
\begin{subsection}{Literature} \label{se:literature}
The analysis in the paper involves Sobolev and Dirichlet spaces on quasidisks.
  In particular we require the work of W. Smith, D. Stegenga \cite{SmithStegenga} and A. Stanoyevitch, D. Stegenga \cite{StanoyevitchStegenga} on Poincar\'e inequalities on John domains. Later in the investigation of the regularity of the WP-class quasicircles we enter the realm of geometric measure theory and utilize a result due to K. Falconer and D. Marsh \cite{FM} characterizing biLipschitz equivalence of quasidisks. Furthermore the study of Sobolev and Besov spaces, and in particular the traces of functions on the boundary of so called $s$-regular domains (to which we show the WP-class quasidisks belong), was investigated by A. Jonsson \cite{Jon}. Finally the existence of the solution to the Dirichlet problem on the WP-class quasidisks, with Besov space boundary data, hinges on results of T.-K. Chang, J. Lewis \cite{ChangLewis} and D. Mitrea, M. Mitrea, S. Monniaux \cite{MMM}.

 The study of WP-class quasicircles was initiated by Cui \cite{Cui}
 and Hui \cite{GuoHui}, in connection with finding a theory of the universal
 Teichm\"uller space based on $L^p$ Beltrami differentials.  The memoir \cite{TT}
 of Takhtajan and Teo obtained wide-ranging results on the WP-class universal
 Teichm\"uller space, including for example sewing formulas for the Laplacian and potentials
 for the Weil-Petersson metric.
 This stimulated a great deal of interest in the subject
 (see Shen \cite{Shen_characterization}).
 Further discussion can be found in Section \ref{se:applications}.
\end{subsection}
\end{section}

\begin{section}{Dirichlet problem on WP quasidisks}
\begin{subsection}{ WP quasidisks and non-overlapping maps}
 \label{se:rigged_moduli_space_def}
 In this section, we give the definition of the refined rigged moduli space and
 related function spaces.

 We will be concerned with quasiconfomally extendible conformal maps of a certain special class.
 Denote the set of Beltrami differentials on $\disk^-$ satisfying
 \[  \iint_{\disk^-} \frac{|\mu(z)|^2}{(|z|^2 - 1)^2} \,dA<\infty  \]
 by $L^2_{\mathrm{hyp}}(\disk^-)$ (``hyp'' stands for ``hyperbolic'', since the condition above
 says that $|\mu|$ is square-integrable with respect to the hyperbolic area element
 on $\disk^-$).  With this in mind, we make the following definition, following terminology
 of \cite{Shen}.
 \begin{definition}
  Let $f:\disk^+ \rightarrow \mathbb{C}$ be a one-to-one analytic map.  We say that $f$ is WP-class
  if it has a quasiconformal extension to $\sphere$ whose Beltrami differential $\mu$ on $\disk^-$
  is in $L^2_{\mathrm{hyp}}(\disk^-)$.   We say that a Jordan curve $\Gamma$ in $\mathbb{C}$ is a WP-class
  quasicircle if there is a WP-class map $f$ taking $\disk^+$ onto the bounded complement of $\Gamma$.
 \end{definition}
 We assume that $\infty \notin \Gamma$ for convenience throughout the paper, although the definition
 above can be naturally extended to allow this possibility.  We refer to each of the two complements of
 a WP-class quasicircle as WP-class quasidisks.

 \begin{definition} \label{de:Oqco}
 Let $p \in \mathbb{C}$.  Let $\Oqco(p)$ denote the set of holomorphic functions $f \colon \disk^+ \to \mathbb{C}$ such
 that $f(p)=0$, $f$ is one-to-one and has quasiconformal extension to $\mathbb{C}$ and
 \[  \iint_{\disk^+} \left| \frac{f''(z)}{f' (z)} \right|^2 \,dA <\infty,  \]
where $dA$ denotes the Euclidean area element.
\end{definition}

 The following theorem was proven by Hui \cite{GuoHui} as a special case of a general result for Beltrami differentials in weighted $L^p$ spaces,  although the proof in $L^2$ goes back to Cui \cite{Cui}.
 \begin{theorem}
  Let $f \colon \disk^+ \to \mathbb{C}$ be a conformal map such that $f(0)=0$.  Then $f \in \Oqco(0)$ if and only if $f$ has a quasiconformal extension $\tilde{f}$ to the plane such that the Beltrami differential $\mu$ of the restriction of $\tilde{f}$ to $\disk^-$ is in $L^2_{\mathrm{hyp}}(\disk^-)$.
 \end{theorem}
 Of course, the result immediately extends to $\Oqco(p)$ for any $p \in \mathbb{C}$ since $f''/f'$ is invariant under
 the addition of a constant to $f$.
 Thus WP-class quasicircle $\Gamma$ are characterized by the property that for any $p$ in the bounded component of $\mathbb{C} \backslash
 \Gamma$, $\Gamma$
 is the boundary of $f(\disk^+)$ for some $f \in \Oqc(p)$.  We summarize this as follows, together with two more characterizations.
 \begin{theorem}  Let $\Gamma$ be a Jordan curve in the plane.  The following are equivalent.
  \begin{enumerate}
   \item $\Gamma$ is a WP-class quasicircle.
   \item  For any $p$ in the bounded component $\Omega^+$ of $\sphere \backslash \Gamma$, $\Omega^+$ is the image of some $f \in \Oqco(p)$.
   \item The unbounded component $\Omega^-$ of $\sphere \backslash \Gamma$ is the image of some one-to-one map $g \colon \disk^- \rightarrow \sphere$
    which is analytic except for a simple pole at $\infty$, such that $g''/g'$ is in $L^2(\disk^-)$.
   \item The unbounded component $\Omega^-$ of $\sphere \backslash \Gamma$ is the image of some one-to-one map $g \colon \disk^- \rightarrow \sphere$
    which is analytic except for a simple pole at $\infty$, such that $\tilde{g} \in \Oqco(0)$ where $\tilde{g}(z)=1/g(1/z)$.
  \end{enumerate}
 \end{theorem}
 This follows from the above discussion and \cite[Theorem 1.12, Chapter II]{TT}, using the invariance of the Schwarzian under the transformation
 $f \mapsto 1/f(1/z)$.

 We will be concerned with collections $n$-tuples of such maps, whose images do not overlap in a certain sense.
 \begin{definition}{\label{de:rigged_normalized}}
  Let $n\geq 2$. We say that $f =(f_0,\ldots,f_n)$ is an \n-rigging if
  $f_i \colon \disk^+ \rightarrow \mathbb{C}$, $i=0,\ldots,n$, $f_n \colon \disk^- \rightarrow \sphere$ and
  the following conditions are satisfied.
  \begin{enumerate}
   \item $f_0 \in \Oqco(0)$, $f_i \in \Oqco(p_i)$ for $i=1,\ldots,n$, and $1/f_n(1/z) \in \Oqco(0)$.
   \item The closures of the images of $f_i$, $i=0,\ldots,n$, are pair-wise disjoint.
  \end{enumerate}
  Denote the set of \n-riggings by $\Rig(n)$.
 \end{definition}
 Note that condition (2) implies that $p_i \neq 0$ for all $i=1,\ldots,n$ and $p_i \neq p_j$
 whenever $i \neq j$.  We will refer to \n-riggings as riggings if there is no need to specify $n$.

 Given $f \in \Rig(n)$ let
 \[  \riem = \sphere \backslash ( \overline{f_0(\disk^+)} \cup \cdots \cup \overline{f_n(\disk^-)}).  \]
 It should be understood throughout the paper that $\riem$ depends on $f$.  Let $D_i = f_i(\disk^+)$ for $i=0,\ldots,n-1$
 and $D_i=f_i(\disk^-)$ for $i=n$. We call $D_i$ the \ith cap
 and $\partial_i \riem = \partial D_i$ the \ith boundary curve.
 Finally, let $\Omega_i = \sphere \backslash D_i$.

\begin{remark} \label{re:Shen_solution_and_extension_to_RS} A WP-class quasisymmetry is a quasisymmetric mapping $\phi:\mathbb{S}^1 \rightarrow \mathbb{S}^1$
such that the corresponding conformal welding pair of conformal maps $f:\disk^+ \rightarrow \mathbb{C}$
 and $g: \disk^- \rightarrow \sphere$ such that $g^{-1} \circ f =\phi$ satisfies $f \in \Oqco(0)$ \cite{GuoHui, TT} .
 The problem of finding a direct characterization was posed by Takhtajan and Teo \cite{TT}.  It was solved recently by
 Shen \cite{Shen_characterization}, who showed that $\phi$ is WP-class if and only if $\phi$ is absolutely continuous
 and $\log{\phi'}$ is in $H^{\frac{1}{2}}(\mathbb{S}^1)$.

 It is not difficult to extend the notion of WP-class quasisymmetry to maps from $\mathbb{S}^1$ to a bordered
 Riemann surface \cite{RSS_Hilbert}, and similarly to extend the notion of $\Oqco(p)$ to a compact Riemann
 surface with distinguished points.  Sewing on disks using WP-class quasisymmetries results of course in
 a compact Riemann surface with non-overlapping maps in $\Oqco(p)$.  In genus zero, we obtain precisely the
 set $\mathcal{R}$. Thus, the situation considered in this paper
 is the most general one for Riemann surfaces with WP-class boundary parameterizations.
 Details can be found in \cite{RSS_Hilbert}.
\end{remark}

 \begin{remark}  The collection $\Rig$ of riggings is closely related to the so-called rigged moduli space of
  punctured Riemann surfaces of genus zero, which appears in conformal field theory.
   That is, we consider the set of compact Riemann surfaces $\riem^P$
   of genus $g$ with $n+1$ punctures $p_0,\ldots,p_n$.  The rigged moduli space $\Mgn$ is the set of equivalence classes
   \[  \Mgn = \{ (\riem,\psi) \}/\sim \]
   where
   \begin{enumerate}
   \item $\psi=(\psi_0,\ldots,\psi_n)$ where $\psi_i:\disk^+ \rightarrow \riem^P$ satisfy $\psi_i(0)=p_i$ for
    $i=0,\ldots,n$  ,
   \item there is a local coordinate $\zeta_i$ of $p_i$  containing the closure of the image of $\psi_i$ such that $\zeta_i \circ \psi_i \in \Oqco(0)$ and $\zeta_i(p_i)=0$  ,
   \item the closures of the images of $\psi_i$ are pair-wise disjoint, and
   \item two pairs $(\riem,\psi)$ and $(\riem',\psi')$ are said to be equivalent if there is a biholomorphism
   $\sigma:\riem \rightarrow \riem'$ such that $\psi'_i=\sigma \circ \psi_i$ for $i=0,\ldots,n$.
   \end{enumerate}

   By the uniformization theorem, every element of $\Mon$ can be represented by an element of $\Rig(n)$,
   since we can assume $\riem^P=\sphere$, $p_0=0$, $p_n=\infty$ and set $f_i=\psi_i$ for $i=0,\ldots,n-1$ and $f_n=\psi_n(1/z)$.
   If we furthermore impose the normalizations $f_n'(\infty)=1$ in the case $n=2$, and $f_1(0)=1$ in the case $n \geq 3$
   on elements of $\Rig$, then we obtain a set of riggings in one-to-one correspondence with $\Mon$.
 \end{remark}
 \begin{remark}
   We have shown that $\Mgn$  is a complex Hilbert manifold in \cite{RSS_Hilbert} closely related to a refined
   Teichm\"uller space of bordered Riemann surfaces of genus $g$ with $n$ boundary curves
   homeomorphic to $S^1$.  If the condition that $\psi_i$ are in $\Oqco$ is weakened to
   quasiconformal extendibility, then the resulting rigged moduli space
   is a Banach manifold \cite{RadnellSchippers_monster}.  Furthermore it is the quotient of the
   Teichm\"uller space of bordered Riemann surfaces modulo a discrete group.
 \end{remark}
\end{subsection}
\begin{subsection}{Function spaces associated with the Dirichlet problem}
 Recall the definition of the holomorphic Dirichlet spaces given in Section \ref{se:results}.  We will also require
 the space of harmonic functions of finite Dirichlet energy on a quasidisk, and the characterization of their boundary values.

In this section, $\Omega$ is always a quasidisk containing $0$ such that $\overline{\Omega}$ does not contain
 $\infty$.  We define the harmonic Dirichlet space
\begin{equation}
 \mathcal{D}_{\mathrm{harm}}(\Omega)  = \left\{ h \colon \Omega \rightarrow \mathbb{C} \,:\,h \ \mbox{harmonic and} \ \iint_{\Omega} | h'|^2 \,dA + \iint_{\Omega} |\overline{h}'|^2 \,dA <\infty \right\}.  \
 \end{equation}

We endow $\mathcal{D}_{\mathrm{harm}}(\Omega)$ with a norm
\begin{equation}\label{defn:norm of harmonic dirichlet}
\| h \|_{\mathcal{D}_{\mathrm{harm}}(\Omega)} = \left\{ |h(0)|^2 + \iint_{\Omega} | h'|^2 \,dA + \iint_{\Omega} |\overline{h}'|^2  \,dA\right\}^{\frac{1}{2}}.
\end{equation}

 Next we recall the definition of Sobolev spaces on open connected subsets of $\mathbb{R}^2$.
\begin{definition}
Let $\Omega$ be an open connected domain in the plane. Denote by $H^{1}(\Omega)$ the Sobolev space of functions in $L^{2}(\Omega)$ with
\begin{equation}\label{defn:norm of sobolev}
\Vert h\Vert_{H^1(\Omega)}:=\{\Vert h\Vert^2_{L^2(\Omega)}+ \Vert h'\Vert^2_{L^2(\Omega)}+\Vert \overline{h}' \Vert^2_{L^2(\Omega)}\}^{\frac{1}{2}}<\infty,
\end{equation}
 where the derivations are in the sense of distributions.
\end{definition}

\begin{theorem}\label{thm: equivalence of the dirichlet and sobolev norms}
Let $\Omega$ be a quasidisk containing $0$. Then for $h\in \mathcal{D}_{\mathrm{harm}}(\Omega)$ one has
\begin{equation}\label{equivalence for dirichletsobolev}
C'\Vert h\Vert_{\mathcal{D}_{\mathrm{harm}}}\leq  \Vert h\Vert_{H^1(\Omega)}\leq C \Vert h\Vert_{\mathcal{D}_{\mathrm{harm}}} .
\end{equation}

\end{theorem}
\begin{proof}
To show that $\Vert h\Vert_{H^1(\Omega)}\leq C \Vert h\Vert_{\mathcal{D}_{\mathrm{harm}}},$ using \eqref{defn:norm of harmonic dirichlet} and \eqref{defn:norm of sobolev}, it would be enough to show that
$\Vert h\Vert^2_{L^2(\Omega)}\leq C (|h(0)|^2+ \Vert h'\Vert^2_{L^2(\Omega)}+\Vert \overline{h}' \Vert^2_{L^2(\Omega)}),$ for $h\in \mathcal{D}_{\mathrm{harm}}(\Omega).$

Now it is well-known, see for example \cite{SmithStegenga} and \cite{StanoyevitchStegenga}, that for any quasidisk $\Omega$ (which is a so called {\it John domain}), for any arbitrary $z_0$ in $\Omega$, and for $F$ holomorphic in $\Omega$, one has the analytic Poincar\'e inequality

\begin{equation}\label{bound on the L2 norm:1}
 \Vert F-F(z_0)\Vert_{L^2(\Omega)}\leq C(z_0)  \Vert F'\Vert_{L^2(\Omega)}.
\end{equation}

Now since $\Omega$ is simply connected, any $h\in \mathcal{D}_{\mathrm{harm}}(\Omega)$ can be represented as $h(z)= F(z)+\overline{G(z)}$, where $F$ and $G$ are holomorphic functions in $\Omega$. Therefore, \eqref{bound on the L2 norm:1} yields

\begin{equation} \label{Poincare inequality}
\begin{aligned}
\Vert h\Vert^2_{L^2(\Omega)} & \leq 2\Vert h-h(0)\Vert^2_{L^2(\Omega)}+ 2|\Omega||h(0)|^2  \\
& \leq 2 \Vert F-F(0)\Vert^2_{L^2(\Omega)}+2\Vert G-G(0)\Vert^2_{L^2(\Omega)}+ 2|\Omega||h(0)|^2 \\
& \leq 2|\Omega||h(0)|^2+C_1(\Vert F'\Vert^2_{L^2(\Omega)}+\Vert G'\Vert^2_{L^2(\Omega)}) \\
& \leq C_2(|h(0)|^2+\Vert h'\Vert^2_{L^2(\Omega)}+ \Vert \overline{h}'\Vert^2_{L^2(\Omega)}),
\end{aligned}
\end{equation}
where we have also used the holomorphicity of $F$ and $G$ which yields that $F'(z)=h'(z)$ and $G'(z)=\overline{h}'$. This concludes the proof of the second estimate in \eqref{equivalence for dirichletsobolev}.

In order to show $C'\Vert h\Vert_{\mathcal{D}_{\mathrm{harm}}}\leq  \Vert h\Vert_{H^1(\Omega)}$, it is enough to show that $|h(0)|^2 \leq C \Vert h\Vert^2_{L^2(\Omega)}$. To this end, we observe that since $0\in\Omega$ and $h$ is harmonic in $\Omega$, there is an $r>0$ such that $\mathbb{D}(0,r)\subset \Omega$ and by the mean-value theorem for harmonic functions one has

\begin{equation}
  |h(0)|\leq \frac{1}{\pi r^2} \iint_{\mathbb{D}(0,r)}|h(z)|\, dA(z)\leq \frac{|\Omega|}{\pi r^2} \iint_{\Omega}|h(z)|\, \frac{dA(z)}{|\Omega|}.
\end{equation}

Now Jensen's inequality yields that

\begin{equation}
|h(0)|^2\leq \frac{|\Omega|}{\pi^2 r^4} \iint_{\Omega}|h(z)|^2\, dA(z)
\end{equation}
which gives us the desired estimate. This ends the proof of the theorem.
\end{proof}

We will need the appropriate function spaces of boundary values corresponding to the
Dirichlet problem.  In the case of a WP-class quasidisks, the most natural space
is a certain Besov space, which we will define shortly.
\begin{definition}
Let $E$ be a compact subset of $\mathbb{R}^2$ and $0 < s \leq 2$. We say that $E$ is \s-regular if it is bounded and if there is a constant $C_E$ such that

\begin{equation}\label{defn:s-regular set}
\frac{1}{C_E }r^s \leq  \mathscr{H}^s (B(x, r)\cap E) \leq C_E r^s
\end{equation}
for all $x\in E,$ $0 < r \leq  \mathrm{diam}(E)$, where $\mathscr{H}^s$ denotes the \s-dimensional Hausdorff measure.
\end{definition}

From this, it also follows that the Hausdorff dimension $d_{\mathrm{H}}$ of an \s-regular set is equal to $s$.
Now, assume that $\Omega$ is a quasidisk in $\mathbb{R}^2$. It is known, by a result of S. Rohde \cite{Rohde} that $\partial \Omega$, i.e., the quasicircle, is bi-Lipschitz equivalent to a snowflake curve. Now since these latter curves are \s-regular for some $s$ and the \s-regularity is known to be invariant under bi-Lipschitz maps, it follows that $\partial\Omega$ is also \s-regular. Furthermore, a result of F. Gehring and J. V\"ais\"al\"a \cite{GV} concerning $d_{\mathrm{H}}(\partial \Omega)$ shows that $s\in [1,2)$.

\begin{definition}
Let $E\subset \mathbb{R}^2$ be an \s-regular set, $0 < s\leq 2$. The Besov space $B^{s/2}_{2,2} (E)$
consists of all $u\in L^2(E)$ for which
\[\int_{E}\int_{E}\frac{|u(x)-u(y)|^2}{|x-y|^{2s}}\, d \mathscr{H}^s(x)\,d \mathscr{H}^s(y) < \infty.\]
The norm of this Besov space is defined by
\begin{equation}\label{defn: besov space norm}
\Vert u\Vert_{B^{s/2}_{2,2} (E)}= \Vert u\Vert_{L^2(E)}+ \left \{\int_{E}\int_{E}\frac{|u(x)-u(y)|^2}{|x-y|^{2s}}\, d \mathscr{H}^s(x)\,d \mathscr{H}^s(y)\right\}^{1/2}.
\end{equation}
\end{definition}

It was shown by A. Jonsson \cite{Jon}, and A. Jonsson and H. Wallin \cite{JW} that if $\Omega$ is a domain in $\mathbb{R}^2$ whose boundary $\partial \Omega$ is an \s-regular set, then the elements of $H^1(\Omega)$ have a well-defined trace or restriction to $\partial \Omega.$ More precisely, given $\Omega \subset \mathbb{R}^2$ whose boundary $\partial \Omega$ is an \s-regular set with $1\leq s <2$, and given $f\in H^1(\Omega),$ its trace $f|_{\partial \Omega}$ exists as a function in the Besov space $B^{s/2}_{2,2} (\partial \Omega),$ moreover for some $C>0$ one has
$\Vert f|_{\partial \Omega}\Vert_{B^{s/2}_{2,2} (\partial \Omega)} \leq C \Vert f\Vert_{H^1(\Omega)}.$ Conversely, every function $f\in B^{s/2}_{2,2} (\partial \Omega)$ (with $1\leq s <2$) can be extended to a function $F \in H^{1}(\Omega)$ in such a way that $F$ depends continuously on the boundary data.

In this paper we will be particularly concerned with WP-class quasicircles which we later show are \1-regular, see Theorem \ref{thm:regularity of WP-class quasicircles}. Thus we single out the Besov space $B^{1/2}_{2,2}(\partial \Omega)$
and define
\begin{equation} \label{eq:Besov_space_definition}
 \mathcal{H}(\partial \Omega)=B^{1/2}_{2,2}(\partial \Omega).
\end{equation}
Note that in the case that $\partial \Omega = \mathbb{S}^1$, this reduces to the space $\mathcal{H}(\mathbb{S}^1)$
of functions with square-integrable half-order derivatives, see e.g. \cite{JW}.  Thus the notation is consistent with that of the introduction.
From this and the discussion above on the regularity of the quasicircles, we can immediately conclude

\begin{proposition}\label{prop: traces on quasidisks}
Given a WP-class quasidisk  $\Omega$ and $u\in H^1(\Omega)$, one has $\Vert u|_{\partial \Omega}\Vert_{\mathcal{H}(\partial \Omega)} \leq C \Vert u\Vert_{H^1(\Omega)}.$ Conversely a function $u\in \mathcal{H}(\partial\Omega)$ has an extension to a function $U\in H^1(\Omega).$
\end{proposition}

We also need to recall the definition of a {\it{chord-arc}} curve.
\begin{definition}
A closed curve in $\mathbb{R}^2$ is called chord-arc if it is a \1-regular quasicircle.
\end{definition}
Therefore $\Gamma$ is a chord-arc curve if and only if it is rectifiable and there is a constant $C>0$ such that the length of the shorter arc of $\Gamma$ joining the two points $w_1$ and $w_2$ is bounded from above by $C|w_1 -w_2|$ (see e.g.\  \cite{GM}).

The following theorem establishes the fact that WP-class quasidisks are {\it{chord-arc domains}}, i.e., bounded domains in the plane that have a chord-arc boundary curve.
This gives a partial answer to the problem of intrinsically characterizing WP-class quasidisks posed by Takhtajan and
 Teo \cite[Part II Remark 1.10]{TT}.  (In the same remark they also posed the problem intrinsically characterizing WP-class quasisymmetries - solved by Shen \cite{Shen_characterization} - mentioned earlier in Remark \ref{re:Shen_solution_and_extension_to_RS}).

\begin{theorem}\label{thm:regularity of WP-class quasicircles}
Let the Beltrami differential $\mu$ belong to $L^2_{\mathrm{hyp}}(\disk^+).$ Then the corresponding WP-class quasicircle is a bi-Lipschitz image of the circle $\mathbb{S}^1$ and hence a chord-arc curve. Furthermore the corresponding WP-class quasidisk is a bi-Lipschitz image of the unit disk $\disk^+$.
\end{theorem}
\begin{proof}  The assumption that the Beltrami coefficient $\mu\in L^2_{\mathrm{hyp}}(\disk^+)$ and the isomorphism $1/z$ between $\disk^+$ and $\disk^-$ together with a result of J. Becker and C. Pommerenke \cite[Corollary 1.4]{BP} yield that the corresponding quasiconformal map $f^{\mu}$ is asymptotically conformal in the sense of \cite {BP}. From this and a result of M. Badger, J. Gill, S. Rohde, T.Toro (see \cite[Corollary 2.7]{BGRT}) it readily follows that $d_{\mathrm{H}}(f^{\mu}(\mathbb{S}^1))=1$, i.e., the quasicircle associated with $\mu\in L^2_{\mathrm{hyp}}(\disk^+)$ has Haussdorff dimension exactly equal to 1. At this point we use a classical result of K. Falconer and D. Marsh \cite{FM} that if two quasicircles have the same Hausdorff dimension then they are bi-Lipschitz homeomorphic. Now since $d_{\mathrm{H}}(\mathbb{S}^1)=1$ it follows that there exists a bi-Lipschitz map $\varphi$ from the plane to itself such that $f^{\mu}(\mathbb{S}^1)=\varphi(\mathbb{S}^1)$. Furthermore, since $\mathbb{S}^1$ is both \1-regular and rectifiable and these two properties are both preserved under bi-Lipschitz homeomorphisms, it follows at once that $f^{\mu}(\mathbb{S}^1)$ is a chord-arc curve. Now it follows from a result of   P. Tukia \cite{Tukia} that $\varphi$ has a bi-Lipschitz extension to a map $\Phi$ with $\Omega= \Phi(\disk^+).$
\end{proof}

Now, using the fact that a WP-class quasidisk is a chord-arc domain, one can show that the Dirichlet problem is solvable on a WP-class quasidisk $\Omega$ with boundary values in $\mathcal{H}(\partial \Omega)$. In fact we show that $\mathcal{H}(\partial \Omega)$ consists precisely of functions
which are boundary values of harmonic functions in $\Omega$.  This firmly establishes its naturality.
\begin{theorem} \label{th:trace_of_Dirichlet}  Let $\Omega$ be a
 WP-class quasidisk such that $\infty \notin \Omega$ and $0\in\Omega$. Then the following statements are valid.
 \begin{enumerate}
   \item Every function $h \in \mathcal{H}(\partial \Omega)$ is the
 trace of an element $H \in \mathcal{D}_{\mathrm{harm}}(\Omega)$; furthermore,
 the linear operator taking $h$ to $H$ is bounded with respect to the $\mathcal{D}_{\mathrm{harm}}(\Omega)$
 and Besov norms.
   \item Every element $H \in \mathcal{D}_{\mathrm{harm}}(\Omega)$ has a trace in $\mathcal{H}(\partial \Omega)$.
    Furthermore the linear operator taking $H$ to its trace is bounded with respect to the $\mathcal{D}_{\mathrm{harm}}$ norm and Besov norms.
 \end{enumerate}

\end{theorem}
\begin{proof} $(1)$ A careful examination of the proof of Theorem 3.4 in \cite{MMM} reveals that the Dirichlet problem on a chord-arc domain with boundary data in $\mathcal{H}(\partial \Omega)$ has a solution whose $H^1$ norm depends continuously on the data; see in particular \cite{ChangLewis}. Since the $H^1$ and $\mathcal{D}_{\mathrm{harm}}$ norms are equivalent by Theorem \ref{thm: equivalence of the dirichlet and sobolev norms}, this proves the first part of the theorem.

 $(2)$ The identity map from $\mathcal{D}_{\mathrm{harm}}(\Omega)$ to $H^1(\Omega)$ is bounded by Theorem \ref{thm: equivalence of the dirichlet and sobolev norms}. The claim now follows from the fact that the trace operator from $H^1(\Omega)$ to $\mathcal{H}(\partial \Omega)$
 is bounded by Proposition \ref{prop: traces on quasidisks}.
\end{proof}

\end{subsection}
\end{section}

\begin{section}{The jump problem and parametrization of the Dirichlet spaces in the simply connected case}
In this section, we show that the jump problem is solvable for WP-quasidisks for boundary values in $\mathcal{H}(\partial \Omega)$.
We then use this result to define a natural Hilbert space isomorphism between $\mathcal{D}(\disk^\pm)$ and $\mathcal{D}(\Omega^\pm)$ for
WP-quasidisks.  This isomorphism is a version of an isomorphism of Shen \cite{ShenFaber}
between the formal space $l^2$ and $\mathcal{D}(\Omega^\pm)$ defined using Faber polynomials.
Here, we identify $l^2$ with $\mathcal{D}(\disk^\pm)$ and give
 this isomorphism an explicit formula in terms of
composition and projection operators.  We deal with the jump problem in Section \ref{se:jump} and the isomorphism in
Section \ref{se:isomorphism_simply_connected}.
\begin{subsection}{The jump problem on WP-class quasicircles} \label{se:jump}

From Theorem \ref{thm:regularity of WP-class quasicircles} it follows immediately that the boundary of a WP-class quasidisk is a rectifiable curve and therefore Cauchy integrals will be defined in a natural way on the WP-class quasicircles.
Next we discuss Cauchy integrals. Let $\Gamma$ be a closed oriented rectifiable Jordan curve in the plane not containing $\infty$ and let $\Omega^+$ and $\Omega^-$ denote its two complementary regions. $\Omega^-$ will denote the region containing $\infty$.  Given a function $f$ on $\Gamma$ one defines its Cauchy integral $P(\Gamma) f(z)$ for $z \notin \Gamma$ by

\begin{equation}\label{defn: Cauchy integral}
P(\Gamma)f(z)=\frac{1}{2\pi i} \int_{\Gamma} \frac{f(\zeta)}{\zeta-z}\, d\zeta.
\end{equation}
Now if $P_{+}(\Gamma)f$ and $P_{-}(\Gamma)f$ are restrictions of $P(\Gamma)f(z)$ to $\Omega^+$ and $\Omega^-$ respectively, and if $f_{+}$ and $f_{-}$ are their boundary values, the Sokhotski-Plemelj jump formula yields that

\begin{equation}\label{defn: sokotzki-plemelj}
f_{\pm}(z)=\frac{\pm 1}{2} f(z) + \frac{1}{2\pi i} \mathrm{P.V.} \int_{\Gamma} \frac{f(\zeta)}{\zeta-z}\, d\zeta, \,\,\,\, z\in \Gamma
\end{equation}

A classical result due to G. David \cite{david} yields that if $\Gamma$ is a chord-arc curve then given $f \in L^2 (\Gamma)$, one has the estimate, $\Vert f_{\pm}\Vert_{L^2 (\Gamma)}\leq C \Vert f\Vert_{L^2 (\Gamma)}$.

We will also need estimates for a certain integral operator that appears frequently in function theory. This operator is defined by

\begin{equation}\label{defn: operator T}
T_{\Omega}f(z)= \iint_{\Omega} \frac{f(\zeta)}{\zeta-z}\, dA(\zeta).
\end{equation}
We also have that
\begin{equation} \label{eq:Tderivative_principal_value}
 \partial_z T_{\Omega}f(z) = \lim _{\varepsilon \to 0}\iint_{\Omega\cap\{ |\zeta-z|> \varepsilon\}} \frac{f(\zeta)}{(\zeta-z)^2}\, dA(\zeta).
\end{equation}

\begin{lemma}\label{Lem: H1 boundedness of operator T}

Let $\Omega$ be a bounded domain in the plane. Then $\Vert T_{\Omega}f\Vert_{L^2 (\Omega)}\leq C \Vert f\Vert_{L^2 (\Omega)}$ and $\Vert \partial_z T_{\Omega}f\Vert_{L^2 (\Omega)}\leq C \Vert f\Vert_{L^2 (\Omega)}$.  Thus

\begin{equation}\label{estim: Sobolev norm of operator T}
 \Vert T_{\Omega}f\Vert_{H^1(\Omega)}\leq C \Vert f\Vert_{L^2 (\Omega)}.
\end{equation}
\end{lemma}

\begin{proof}
To establish the boundedness of $T_{\Omega}$ on $L^2 (\Omega)$, we observe that using the Cauchy-Schwarz inequality we have the pointwise estimate

\begin{equation}\label{estim:Pointwise estimate for operator T}
|T_{\Omega}f(z)| \leq \left\{\iint_{\Omega} \frac{|f(\zeta)|^2}{|\zeta-z|}\, dA(\zeta)\right\}^{1/2}\left\{\iint_{\Omega} \frac{1}{|\zeta-z|}\, dA(\zeta)\right\}^{1/2}.
\end{equation}

But if $|\Omega|$ denotes the 2-dimensional Lebesgue measure of $\Omega$, then it can be shown, see e.g., W. Tutschke \cite{Tut} that $\iint_{\Omega} \frac{1}{|\zeta-z|}\, dA(\zeta)\leq 2\sqrt{\pi}|\Omega|^{1/2},$ for $z \in \Omega.$ Therefore squaring and integrating \eqref{estim:Pointwise estimate for operator T} and using Fubini's theorem, we obtain
\begin{equation}\label{estim: L2 estimate for operator T}
\iint_{\Omega} |T_{\Omega}f(z)|^2 dA(z)\leq 4\pi|\Omega|\Vert f\Vert^{2}_{L^{2}(\Omega)},
\end{equation}
which is the desired $L^2$ boundedness.

In order to show the $L^2 (\Omega)$ boundedness of $\partial_z T_{\Omega}$ we just use
(\ref{eq:Tderivative_principal_value}) and the fact that $$\Vert \partial_z T_{\Omega}f\Vert_{L^2 (\Omega)} \leq \Vert \partial_z T_{\Omega}f\Vert_{L^2 (\mathbb{C})}=\Vert \partial_z T_{\mathbb{C}}(f\chi_{\Omega})\Vert_{L^2 (\mathbb{C})},$$
where $\chi_{\Omega}$ denotes the characteristic function of $\Omega$. Now, as was shown by L. Ahlfors in \cite{Ahlfors}, $\Vert \partial_z T_{\mathbb{C}}f\Vert_{L^2(\mathbb{C})}=\pi \Vert f\Vert_{L^2 (\mathbb{C})}$. Therefore

\begin{equation}
\Vert \partial_z T_{\Omega}f\Vert_{L^2 (\Omega)} \leq \pi \Vert f\chi_{\Omega}\Vert_{L^2 (\mathbb{C})} =\pi\Vert f\Vert_{L^2 (\Omega)},
 \end{equation}
 which concludes the proof of the lemma.
\end{proof}
Now we have all the ingredients to state and solve the following Riemann boundary value problem, sometimes called (with various kinds of regularity) the jump problem.

\begin{theorem} \label{th:Cauchy_bounded}
Let $\Omega^+$ be a WP-class quasidisk as above and let $u$ be in $\mathcal{H}(\partial\Omega^+)$.  Let $\Omega^-$ denote the complement of
$\overline{\Omega^+}$ in $\sphere$.  Then the jump problem can be solved with $u$ as data in the sense that there exist holomorphic functions $u_{\pm}$ on $\Omega^{\pm}$ such that, $u_\pm \in \mathcal{D}(\Omega^{\pm})$ and $u_{+}(z)-u_{-}(z)= u(z)$ for $z\in \partial\Omega$. Furthermore $u_\pm$ depend continuously on the data; that is
the Cauchy projections are bounded.
\end{theorem}
\begin{proof}
From the discussion above on the Cauchy integral, it readily follows that the solution of this problem is given by  $u(z)_\pm= P_{\pm}(\partial\Omega^+) u(z)$. It remains to prove that one has the estimate

\begin{equation}\label{jump problem estimate}
\Vert u_+ \Vert_{H^1(\Omega^{+})} \leq c \Vert u \Vert_{\mathcal{H}(\partial\Omega^+)}.
\end{equation}
The corresponding estimate on $\Omega^{-}$ is similar.
Now since $u\in \mathcal{H}(\partial\Omega^+)$, it has an extension $v \in H^1(\Omega^{+})$ (actually this $v$ also has an extension to the whole plane thanks to the result in \cite{GLV}). Furthermore, Proposition \ref{prop: traces on quasidisks} yields that $\Vert v \Vert_{H^1(\Omega^{+})} \leq c \Vert u \Vert_{\mathcal{H}(\partial\Omega^+)}.$ Moreover it is known that for $v\in H^1(\Omega^{+})$ (using the fact that $\partial\Omega^+$ is rectifiable),

\begin{equation}\label{Cauchy integral for H1}
P_{+}(\partial\Omega^+) v(z)= v(z)+\frac{1}{\pi} \iint_{\Omega^{+}} \frac{\overline{\partial}v(\zeta)}{(\zeta-z)^2}\, dA(\zeta)=v(z)+\frac{1}{\pi}T_{\Omega^{+}}(\overline{\partial}v)(z) ,
\end{equation}
where the integral above is taken as a principal value integral.

Using these facts and estimate \eqref{estim: Sobolev norm of operator T} of Lemma \ref{Lem: H1 boundedness of operator T} we can deduce that

\begin{align*}
              \Vert u_+\Vert_{H^1(\Omega^{+})} = \Vert P_{+}(\partial\Omega^{+}) v \Vert_{H^1(\Omega^{+})}
              & \leq \Vert v \Vert_{H^1(\Omega^{+})}+ C_1 \Vert \overline{\partial}v\Vert_{L^2(\Omega^{+})} \\
               &\leq  (1+C_1)  \Vert v\Vert_{H^1(\Omega^{+})} \\
               & \leq C_2  \Vert u\Vert_{\mathcal{H}(\partial\Omega^+)}
\end{align*}
as claimed.
\end{proof}
 Now, as a corollary of Theorem \ref{th:Cauchy_bounded} we have the following result.

\begin{corollary}\label{cor: traces of holomorphic functions in dirichlet space}
 Let $\Omega^+$ be a WP-class
quasidisk in the plane such that $\infty \notin \overline{\Omega^+}$, bounded by the curve $\Gamma$. Then the operators
$ P_{\pm}(\Gamma) \colon \mathcal{H}(\Gamma) \rightarrow \mathcal{D}(\Omega^\pm)$ are bounded.
\end{corollary}

Note that because of the limiting behaviour of the Cauchy kernel as $z \rightarrow \infty$, we have that $P_{-}(\Gamma) h(z) \rightarrow 0$ as $z \rightarrow \infty$, so $P_{-}(\Gamma)$ does
map into $\mathcal{D}(\Omega^-)$.
\end{subsection}
\begin{subsection}{The Dirichlet space isomorphism for WP quasidisks} \label{se:isomorphism_simply_connected}
 In this section we define a natural isomorphism between $\mathcal{D}(\disk^\pm)$ and
 $\mathcal{D}(\Omega^\pm)$ using Faber polynomials.  This isomorphism is related to an isomorphism
 of Shen \cite{ShenFaber} under an identification of $\mathcal{D}(\disk^\pm)$ with $l^2$.

 First, we define the Faber polynomials.   We restrict their definition to
 WP-class quasidisks for convenience; however, they can be (and usually are) defined in greater generality.
 \begin{definition} \label{de:Faber_polynomials}
  Let $\Omega^+$ be a WP-class quasidisk whose closure does not contain $\infty$, and let $\Omega^-$ be
  the complement in $\sphere$ of its closure. Let $p$ be a fixed point in $\Omega^+$.  Let $F^- \colon \disk^- \rightarrow \Omega^-$ be a conformal map such that $F^-(\infty)=\infty$ (that is, $F^-$ is one-to-one and onto, and holomorphic except for a simple pole at $\infty$).    Let $F^+ \colon \disk^+ \rightarrow \Omega^+$ be a conformal map such that $F^+(0)=p$.
  For $k \in \mathbb{Z}$ and $k \geq 0$, the $k$th Faber polynomial of $\Omega^+$ is
  \begin{equation} \label{eq:Faber_defn_in}
   \Phi_k(\Omega^+) = P_+(\partial \Omega)\, C_{(F^-)^{-1}}(z^k).
  \end{equation}
  where $C_{(F^-)^{-1}}$ is composition by $(F^-)^{-1}$.
  For $k > 0$ the $k$th Faber polynomial of  $\Omega^-$  is defined by
  \begin{equation} \label{eq:Faber_defn_out}
   \Phi_k(\Omega^-) = P_-(\partial \Omega)\, C_{(F^+)^{-1}}(z^{-k}).
  \end{equation}
 \end{definition}
 \begin{remark}
 Note that the polynomials depend on both the domains and the choice
 of conformal map (which is only unique up to the argument of the derivative
 at $0$ or $\infty$ respectively).
 We will usually drop the argument $\Omega^\pm$ in $\Phi_k$ when the domain is
 clear from context.
 Since $\Omega^+$ is a WP-class quasidisk, $F^\pm$ extend to quasisymmetries between
 $\partial \disk^\pm$ and $\partial \Omega^\pm$, and similarly for their inverses.
 \end{remark}
 \begin{remark}  By expanding in Laurent series it can be shown that the Faber polynomials $\Phi_k(\Omega^\pm)$ are indeed polynomials in $z$ and $1/z$ respectively.  For example, for $k\geq 0$ the Laurent expansion of $C_{F^-} (z^k)$ contains only finitely many positive powers.  For $k<0$,  there are no positive powers.   Similarly for the Faber polynomials on $\Omega_-$ (see e.g.\   Jabotinsky \cite{Jabotinsky}).
 \end{remark}
 \begin{remark} \label{re:on_Faber_definition}
 Often in the definition of the Faber polynomials the projection $P_\pm(\partial \Omega)$ is simply replaced by truncation (see e.g.\  \cite{Jabotinsky}).    Using a Cauchy integral is of course also classical (e.g., Tietz \cite{Tietz}
 where the curve is assumed to be analytic).

 Another common
 approach is to define the Faber polynomials via a generating function \cite{Durenbook, Pommerenkebook}; for the equivalence see for example \cite{Jabotinsky}.   Finally, note that one can define the Faber polynomials assuming only that $F^\pm$ is analytic and
 one-to-one in a neighbourhood of $0$ or $\infty$ respectively (where ``one-to-one in a neighbourhood of $\infty$''
 means that $F^-$ has a simple pole at $\infty$).

 The definitions given in (\ref{eq:Faber_defn_in}) (\ref{eq:Faber_defn_out}) above have the advantage that
 for any WP-quasidisk they can
 be seen as the restriction of a bounded linear isomorphism on the entire
 Dirichlet space (Theorem \ref{th:Faber_series_decomposition} and Corollary \ref{co:I_single_isomorphism} ahead).
  This conclusion requires our work in Section \ref{se:jump}.
 \end{remark}
 \begin{theorem} \label{th:one-sided_composition}
  Let $\Omega_1$ and $\Omega_2$ be WP-class quasidisks containing
  $0$ and such that $\infty \notin \partial \Omega_i$.  If $F \colon \Omega_1 \rightarrow \Omega_2$ is a conformal map taking $0$ to $0$, and $C_F \colon \mathcal{H}(\partial \Omega_2) \rightarrow \mathcal{H}(\partial \Omega_1)$ is composition by the trace of $F$ on the boundary (i.e., $h \mapsto h \circ F$), then $C_F$ is a bounded map.
 \end{theorem}
 \begin{proof}
 First, observe that since $\Omega_1$ and $\Omega_2$ are quasidisks, $F$ has a
 continuous extension to $\partial \Omega_1$ which is a homeomorphism of $\Omega_1$
 onto $\Omega_2$.

  Let $h \in \mathcal{H}(\partial \Omega_2)$.  By Theorem \ref{th:trace_of_Dirichlet} $H$ has a harmonic extension depending
  continuously on $h$.  Now since $H \mapsto H \circ F$ is an isometry in $\mathcal{D}_{\mathrm{harm}}(\Omega)$, Theorem \ref{th:trace_of_Dirichlet} yields the claim.
 \end{proof}
 \begin{remark}  Note $F$ extends to a quasisymmetry of the boundary.  Since
  $C_{F^{-1}}=C_F^{-1}$, we have that $C_F$ has a bounded inverse under the hypotheses of the theorem.
 \end{remark}
 \begin{remark}
  Theorem \ref{th:one-sided_composition} demonstrates that the Besov space $\mathcal{H}(\partial \Omega)$ is
  in a certain sense a conformal invariant.  This is of course related to the fact that the set of harmonic
  functions of finite Dirichlet energy is a conformal invariant.
 \end{remark}

 Next, we observe some classical identities for the Faber polynomials.  We will use the convenient ``power matrix''
 notation.  Given a function  $\hat{F}^+$ which is analytic near zero and satisfies $\hat{F}(0)=0$ and $\hat{F}'(0)\neq 0$ then define the matrix coefficients $[\hat{F}^+]^m_k$ by
 \[ \hat{F}^+(z)^m= \sum_{k=m}^\infty [\hat{F}^+]^m_k z^k  \]
 for any integer $m$.
 If $m$ denotes the row number and $k$ denotes the column number, then the matrix
 is upper triangular (and doubly infinite) and the product of the
 matrices $[ \hat{F}^+]^m_k$ satisfies
 \begin{equation} \label{eq:power_matrix_homo}
  [\hat{F}^+ \circ \hat{G}^+]^m_l = \sum_{k=l}^m [\hat{F}^+]^m_k [\hat{G}^+]^k_l
 \end{equation}
 Similarly,  given a function  $\hat{F}^-$ which is analytic near $\infty$ except for a simple pole at $\infty$, that is
 $\hat{F}^-(\infty)=\infty$ and $(\hat{F}^-)'(\infty)\neq 0$, then for any integer $m$  define $[\hat{F}^-]^m_k$ by
 \[    \hat{F}^-(z)^m= \sum_{k=-\infty}^m [\hat{F}^-]^m_k z^k.  \]
 The matrix corresponding to $\hat{F}^-$ is lower triangular and
 \[  [\hat{F}^+ \circ \hat{G}^+]^m_l = \sum_{k=m}^l [\hat{F}^+]^m_k [\hat{G}^+]^k_l.  \]

 Now consider a map $F^+$ onto a WP-quasidisk $\Omega$ such that $F^+(0)=p$ and $(F^+)'(0)\neq 0$.  Let $\hat{F}^+(z)=F^+(z)-p$.
 It is easily computed that for $k < 0$ (denoting by $v^k$ the function $z \mapsto z^k$),
 \begin{align} \label{eq:conjugation_identity_p}
  C_{F^+} \, P_-(\partial \Omega) \, C_{(F^+)^{-1}} (v^k) & =  C_{F^+} \, \Phi_k(\Omega^+)  \nonumber\\
  & =  C_{F^+} \left( \sum_{l=k}^{-1} \left[ (\hat{F}^+){-1} \right]^k_l (z-p)^l \right) =
  \sum_{l=k}^{-1}  \left[ (\hat{F}^+){-1} \right]^k_l \sum_{m=l}^\infty  \left[ \hat{F}^+ \right]^l_m z^m  \nonumber\\
  & =  \sum_{l=k}^\infty \sum_{m=l}^\infty \left[(\hat{F}^+)^{-1}\right]^k_l \left[ \hat{F}+ \right]^l_m z^m
    - \sum_{l=0}^\infty \sum_{m=l}^\infty \left[(\hat{F}^+)^{-1}\right]^k_l \left[ \hat{F}^+ \right]^l_m z^m \nonumber \\
  & =  z^k - \sum_{m=0}^\infty \sum_{l=0}^m  \left[(\hat{F}^+)^{-1}\right]^k_l \left[ \hat{F}^+ \right]^l_m z^m
 \end{align}
 where in the last equality we have used the fact that the product of the matrices
 $[(\hat{F}^+)^{-1}][\hat{F}^+]$ is the identity by (\ref{eq:power_matrix_homo}).  Because $\Phi_k(\Omega^+)$ is a polynomial,
 the projection $P_-(\partial \Omega)$ need not refer to the domain $\Omega$.

 A similar computation for $F^-$ satisfying $F^-(\infty)=\infty$ and $(F^-)'(\infty)\neq 0$ shows
 that for $k \geq 0$ (again denoting by $v^k$ the map $z \mapsto z^k$),
 \begin{align} \label{eq:conjugation_identity_infty}
  C_{F^-} \, P_+(\partial \Omega) \, C_{(F^-)^{-1}} (z^k) & =  C_{F^-} \, \Phi_k(\Omega_-) \nonumber \\
  & =  z^k - \sum_{m=-\infty}^{-1} \sum_{l=m}^{-1} [(F^-)^{-1}]^k_l [F^-]^l_m z^m.
 \end{align}
 \begin{remark} \label{re:domain_of_projection_extraneous}
  In fact, (\ref{eq:conjugation_identity_p}) and (\ref{eq:conjugation_identity_infty}) hold for
  arbitrary analytic maps which are one-to-one near $0$ and $\infty$ respectively,
  if rather than applying $P_{\pm}(\partial \Omega)$ one truncates the Laurent series near
  $0$ or $\infty$ appropriately.
 \end{remark}

 The equations (\ref{eq:conjugation_identity_p}) and (\ref{eq:conjugation_identity_infty}) imply the following
 form of a classical identity.
 \begin{proposition} \label{pr:classical_Grunsky_Faber_proj_identity}
  Let $F^+ \in \Oqco(p)$.  Then
  \[  P_+(\mathbb{S}^1) \,C_{F^+} \, P_-(\partial \Omega) \, C_{(F^+)^{-1}} = \mathrm{Id}  \]
  on $\mathcal{D}(\disk^-)$.
  Furthermore, $C_{F^+} \, P_-(\partial \Omega) \, C_{(F^+)^{-1}}$ is independent of $p$
  and $(F^+)'(0)$.

  Similarly, if $F^-$ is a map such that $1/F^-(1/z) \in \Oqco(0)$, then
  \[ P_-(\mathbb{S}^1) \,C_{F^-} \, P_+(\partial \Omega) \, C_{(F^-)^{-1}}  = \mathrm{Id} \]
  on $\mathcal{D}(\disk^+)$ and
  $C_{F^-} \, P_+(\partial \Omega) \, C_{(F^-)^{-1}}$ is independent of $(F^-)'(\infty)$.
 \end{proposition}
 \begin{proof}  This follows immediately from (\ref{eq:conjugation_identity_p}) and (\ref{eq:conjugation_identity_infty})
  and the fact that polynomials are dense in the Dirichlet space.
 \end{proof}

 \begin{theorem} \label{th:Faber_series_decomposition}
  Let $\Omega_{\pm}$ and $F_{\pm}$ be as in Definition \ref{de:Faber_polynomials}.
  Let $h \in \mathcal{D}(\Omega_{\pm})$.  Let $h_n$ be defined by
  the Fourier series of $h \circ F^{\mp}$ as follows:
 \[ h \circ F^{\mp}(e^{i\theta}) =\begin{cases} \sum_{n= -\infty}^\infty h_n e^{in\theta}&\quad \mathrm{if } \, h \in \mathcal{D}(\Omega^+)  \\
  \sum_{n= -\infty }^{\infty} h_{-n}e^{-in\theta}&\quad  \mathrm{if } \, h \in \mathcal{D}(\Omega^-).
\end{cases} \]
  Then $h$ has a Faber series
  \[  h(z) = \sum_{n=N}^\infty h_n \Phi_{n}(\Omega^{\pm})(z),  \]
  (where $N=0$ for $\mathcal{D}(\Omega^+)$ and $N=1$ for $\mathcal{D}(\Omega^-)$)
  which converges uniformly to $h$ on compact subsets of $\Omega_{\pm}$.
 \end{theorem}
 \begin{proof}  We give the proof in the case that $h \in \mathcal{D}(\Omega^-)$.
  The other case differs only notationally.  We have that $h \in \mathcal{H}(\partial\Omega)$ by Theorem \ref{th:trace_of_Dirichlet} and $h \circ F^- \in \mathcal{H}(\partial\disk^+)$ by Theorem \ref{th:one-sided_composition}.  So $h \circ F^-$ indeed has a Fourier
  series
  \[  h \circ F^-(e^{i\theta}) = \sum_{n=-\infty}^\infty h_n e^{i n \theta}.  \]
  Since this series obviously converges in $\mathcal{H}(\partial\disk^+)$ and $C_{{F^-}^{-1}}$ is a continuous map, we have that
  the series
  \[  h(w) = \sum_{n=-\infty}^\infty h_n (F^-)^{-1}(w)^n  \]
  converges in $\mathcal{H}(\partial\Omega)$.  Applying the projection $P_+ (\partial \Omega^+)$ to both sides, which is a continuous map to $\mathcal{H}(\Omega^+)$ by Theorem \ref{th:Cauchy_bounded}, we have
  that
  \begin{align} \label{eq:temp_Faber_series}
   h(w) & = \sum_{n=-\infty}^\infty h_n P_+ (F^-)^{-1}(w)^n \nonumber \\
   & =  \sum_{n=0}^\infty h_n \Phi_n(\Omega)(w).
  \end{align}
  where we have used the fact that $P_+(\partial \Omega^+) h = h$ since $h \in \mathcal{D}(\Omega^+)$.

  We only so far have that this series converges in $\mathcal{H}(\partial\Omega^+)$.  However,
  this implies that the series converges in $\mathcal{D}(\Omega^+)$.  Since
  convergence in the Dirichlet space implies uniform convergence on
  compact subsets of $\Omega^+$ (by representing the sequence of functions using the
  reproducing kernel and applying the Cauchy-Schwarz inequality), the proof is complete.
 \end{proof}

 The Faber polynomials can be used to provide a trivialization of the Dirichlet spaces. Let $\Omega^\pm$ and $F^\pm$ be as in Definition \ref{de:Faber_polynomials}.
 Define
 \begin{equation}  \label{eq:isomorphism_simple_definition}
  \mathfrak{I}(\Omega^\pm, F^\mp) = P_{\pm}(\partial \Omega^\pm) \circ C_{(F^\mp)^{-1}} \colon \mathcal{D}(\disk^\pm) \rightarrow \mathcal{D}(\Omega^\pm).
 \end{equation}
  Note that the isomorphism $\mathfrak{I}(\Omega^\pm,F^\mp)$ has the following property:
 \begin{equation} \label{eq:zn_to_Faber}
 \begin{aligned}
  \mathfrak{I}(\Omega^+,F^-) (z^n) & = \Phi_n(\Omega^+) \quad \text{for } n \geq 0, \text{ and} \\
  \mathfrak{I}(\Omega^-,F^+)(z^n) & =  \Phi_n(\Omega^-) \quad \text{for }  n <0.
 \end{aligned}
 \end{equation}
 \begin{remark}  To avoid notational clutter, we will not explicitly write the
  restriction operator from $\mathcal{D}(\Omega^\pm)$ to $\mathcal{H}(\partial \Omega^+)$.
 \end{remark}
 \begin{corollary} \label{co:I_single_isomorphism} For $\Omega^\pm$ and $F^{\pm}$ as in Definition \ref{de:Faber_polynomials}, $\mathfrak{I}(\Omega^\pm,F^\mp)$ is a bounded linear isomorphism.
 \end{corollary}
 \begin{proof}
  Each operator in the definition of $\mathfrak{I}(\Omega^\pm)$ is bounded, including
  the implicit trace to the boundary, by  Corollary \ref{cor: traces of holomorphic functions in dirichlet space} and Theorems \ref{th:Cauchy_bounded} and \ref{th:one-sided_composition}. Thus $\mathfrak{I}(\Omega^\pm)$ is bounded.

  We only need show that $\mathfrak{I}(\Omega^\pm,F^\mp)$ is a bijection.  We prove
  the claim for $\mathfrak{I}(\Omega^+,F^-)$; the other case is similar.
  Observe that for any $h(z) = \sum_{n=0}^\infty h_n z^n \in \mathcal{D}(\disk^+)$
  \[ \mathfrak{I}(\Omega^\pm,F^-) h = \sum_{n=0}^\infty h_n \Phi_n(\Omega^+).  \]

  To establish surjectivity, we use the density of polynomials in the Dirichlet space (this can
  be seen by using the well-known density of polynomials in $L^2$ \cite{Markusevic} and the definition of
  the Dirichlet norm). It suffices
  to show that the image contains every polynomial.  Since every Faber polynomial $\Phi_k(\Omega^+)$ is
  of degree $k$, any
  polynomial $p$ of degree $n$ or less has a unique expression as a sum of Faber polynomials,
  say $p= \sum_{k=0}^n p_k \Phi_k(\Omega^+)$.  Thus $p$ is the image of $\sum_{k=0}^n p_n z^n$ under
  $\mathfrak{I}(\Omega^+,F^-)$.

  Injectivity is established as follows.  Assume that $\mathfrak{I}(\Omega^+,F^-)(h)=0$, say,
  with $h=\sum_{n=0}^\infty h_n z^n$.
  By Theorem \ref{th:one-sided_composition}, $C_{F^-}$ is a bounded operator with bounded
  inverse $C_{(F^-)^{-1}}$, and thus  $C_{F^-} \mathfrak{I}(\Omega^+,F^-)(h)=0$.
  By equation (\ref{eq:conjugation_identity_infty}), the formula for $C_{F^-} \mathfrak{I}(\Omega^+,F^-)(h)$
  in the basis $z^n=e^{in\theta}$ is (for $k \geq 0$)
  \[  C_{F^-} \, P_+(\partial \Omega) \, C_{(F^-)^{-1}} (h)
  = \sum_{k=0}^\infty \left( h_k z^k -  h_k \left[\sum_{m=-\infty}^{-1} \sum_{l=m}^{-1} [(F^-)^{-1}]^k_l [F^-]^l_m z^m \right]\right) ,
  \]
  and thus since the coefficients of $z^k$ must be zero, $h_k =0$ for all $k \geq 0$.
   Therefore  $\mathfrak{I}(\Omega^+,F^-)$ is a bijection.
 \end{proof}

 As noted earlier, the Faber polynomials depend implicitly on the choice of map $F^\pm$, and thus so does the isomorphism.
 \begin{remark} \label{re:Shen_comparison}
  Shen \cite{ShenFaber} defined an isomorphism from $l^2$ into $\mathcal{D}(\Omega^+)$.
  Equations (\ref{eq:isomorphism_simple_definition}) and (\ref{eq:zn_to_Faber})
  give this isomorphism a function-theoretic meaning.  The isomorphisms
 agree
 under the following identification of $l^2$ and $\mathcal{D}(\disk^+)$:
 \[  (\lambda_1,\ldots) \longmapsto \sum_{n=1}^\infty \frac{\lambda_n}{\sqrt{n}} z^n \]
 except that we include a constant term in our isomorphism for domains not containing
 $\infty$. The proof of the isomorphism here differs from that of Shen.

  However, Shen's isomorphism is more general, and in fact he
 shows that the isomorphism between $l^2$ and $\mathcal{D}(\Omega^+)$ holds even for general quasidisks $\Omega^+$.
This raises
 the following interesting question: can a formula for the isomorphism similar to (\ref{eq:isomorphism_simple_definition})
 be established for a general quasidisk?  Since a quasicircle is not rectifiable in general, it is not immediately
 obvious what could replace the Cauchy projection; nor is it clear what would replace the boundary values of a
 function in the Dirichlet space.  Nevertheless, Shen's result and the quasi-invariance of the Dirichlet
 energy suggest the possibility of such a formula, perhaps with measure-valued boundary values.  See also
 Remark \ref{re:Grunsky_on_quasidisks} ahead.
 \end{remark}
\end{subsection}
\end{section}
\begin{section}{Representation of $\mathcal{D}(\riem)$ by Grunsky matrices}
\begin{subsection}{Dirichlet space isomorphism, multiply-connected case}
 Let $f \in \Rig$ and $\riem$ be the complement of the union of the closure of the images as in Section \ref{se:rigged_moduli_space_def}.  In this section we give a representation of $\mathcal{D}(\riem)$.  This
 can be viewed as a trivialization of this function space, viewed as a
 fiber space over the rigged moduli space. Let
 $\Omega_i$ and $D_i$ be as in Section \ref{se:rigged_moduli_space_def}.

\begin{proposition} \label{prop:multi_trace}
  Any harmonic $h \in \mathcal{D}_{\mathrm{harm}}(\Sigma)$ has a trace on the $i$th boundary curve $\partial \Omega_i$. This trace is in $\mathcal{H}(\partial \Omega_i)$. Furthermore, the operator $T_i \colon \mathcal{D}_{\mathrm{harm}}(\Sigma) \rightarrow \mathcal{H}(\partial \Omega_i)$ taking any function $h$ to its trace on $\partial \Omega$ is bounded.
 \end{proposition}

\begin{proof}
Since by Theorem \ref {thm:regularity of WP-class quasicircles} the boundary components $\partial \Omega_i$ are all chord-arc curves the result is a consequence of Proposition 2.2 in \cite{ChangLewis} and Theorem 3.4 in \cite{MMM}.
\end{proof}

\begin{remark}
This proposition is also valid for $h \in \mathcal{D}(\Sigma)$ i.e. for holomorphic functions in the Dirichlet space. Indeed, every holomorphic function is harmonic and the norm of the harmonic Dirichlet space defined in \eqref{defn:norm of harmonic dirichlet} coincides with the norm of the usual (i.e. holomorphic) Dirichlet space.
\end{remark}

 We define the map
 \begin{align*}
  \mathfrak{K} \colon \mathcal{D}(\Sigma) & \longrightarrow  \mathcal{D}(\Omega_0) \oplus \cdots \oplus \mathcal{D}(\Omega_n) \\
  h & \longmapsto  ( P(\partial_0\riem)_- \, T_0 \, h, \cdots,  P(\partial_{n-1}\riem)_- \, T_{n-1} \, h, P(\partial_n \riem)_+ \, T_n\, h  )
 \end{align*}
 where $T_i$ denotes the trace of the function on the $i$th boundary curve.
 Here, all of the projection maps are negative except the last; thus, they all map
 onto $\mathcal{D}(\Omega_i)$.
By Corollary \ref{cor: traces of holomorphic functions in dirichlet space} and Theorem \ref{prop:multi_trace}, $\mathfrak{K}$ is a bounded linear map. By the Cauchy integral formula, if $\mathfrak{K}(h) = (h_0,\ldots,h_n)$
 then for $z \in \Sigma$,
 \[  h(z) = h_0(z) + \cdots +h_n(z).  \]  Thus $(h_0,\ldots,h_n) \mapsto \left.
  (h_0 + \cdots +h_n) \right|_{\Sigma}$ is a left inverse of $\mathfrak{K}$, so $\mathfrak{K}$ is injective.  On the other hand, since $h_i$ is holomorphic on
  $D_j$ for $i \neq j$, by the Cauchy integral formula
  \[  P(\partial_j \riem) h_i (w) = \frac{1}{2\pi i}\int_{\partial \Omega_j} \frac{h_i(z)}{z-w} \,dz =0  \]
  so for any $(h_0,\ldots,h_n) \in \mathcal{D}(\Omega_0) \oplus \cdots \oplus \mathcal{D}(\Omega_n)$,
  \[  (P(\partial_j \riem) (h_0 + \cdots + h_n)) (w) = h_j(w)  \]
  and thus
  \[  (h_0,\ldots,h_n) = \mathfrak{K} \left(\left. (h_0 + \cdots +h_n)\right|_{\Sigma}
  \right).  \]
  That is, $\mathfrak{K}$ is surjective.   By the open mapping theorem we have just proved
 \begin{theorem} \label{th:K_isomorphism} For $f \in \mathcal{R}$, $\mathfrak{K}$ is a bounded linear isomorphism and
  \[  \mathfrak{K}^{-1}((h_0,\ldots,h_n)) = \left. (h_0 + \cdots +h_n) \right|_{\Sigma}.  \]
 \end{theorem}

 We can now define a trivialization of $\mathcal{D}(\Sigma)$.  Recall the direct sums $\mathcal{D}^\pm$ defined by (\ref{eq:Dpm_definition}).
  Let
 \[ \mathbf{I}_f \colon \mathcal{D}^- \rightarrow \mathcal{D}(\Sigma) \]
 be defined by
 \[   \mathbf{I}_f= \mathfrak{K}^{-1} \circ \left( \mathfrak{I}(\Omega_0,f_0) \oplus \cdots \oplus \mathfrak{I}(\Omega_n,f_n) \right) \]
 where $\mathcal{D}^-$ is defined by (\ref{eq:Dpm_definition}).
 Note that for $i=0,\ldots,n-1$, $\mathfrak{I}(\Omega_i,f_i)$ is of the type
 $\mathfrak{I}(\Omega^-,F^+)$, and for $i=n$ it is of the type $\mathfrak{I}(\Omega^+,F^-)$.  It follows from Corollary \ref{co:I_single_isomorphism} and Proposition \ref{th:K_isomorphism} that
 \begin{theorem} \label{th:I_multi_isomorphism}
  For $f \in \mathcal{R}$, $\mathbf{I}_f$ is a bounded linear isomorphism.
 \end{theorem}

 The intermediate maps above were useful in establishing that $\mathbf{I}_f$ is a bounded isomorphism; however, the picture is clearer if we remove the scaffolding.  Consider rather the maps
 \begin{align*}
  \mathcal{C}_{f^{-1}} \colon \mathcal{H} & \longrightarrow  \mathcal{H}(\partial \riem) \\
  (h_0,\ldots,h_n) & \longmapsto  (h_0 \circ f_0^{-1}, \ldots, h_n \circ f_n^{-1})
 \end{align*}
 where $\mathcal{H}(\partial \riem) = \mathcal{H}(\partial_0 \riem) \oplus \cdots \oplus \mathcal{H}(\partial_n \riem)$
 as in (\ref{eq:H_many_definition})
 and we have suppressed the trace operators from $\mathcal{D}(\disk^\pm)$ to $\mathcal{H}(\mathbb{S}^1)$, and
  the Cauchy projection $\mathcal{P}(\riem):\mathcal{H}(\partial \riem) \rightarrow \mathcal{D}(\riem)$
 as in equation (\ref{eq:Priem_definition}).
 We then have that $\mathbf{I}_f$ has the more transparent form
 \[  \mathbf{I}_f= \left. P(\riem) \,\mathcal{C}_{f^{-1}} \right|_{\mathcal{D}^-} \]
 which agrees with the definition (\ref{eq:I(f)_definition}) in the introduction.
\end{subsection}
\begin{subsection}{Representation of boundary values of $\mathcal{D}(\riem)$ in $\mathcal{H}$ by Grunsky operators}
\label{se:Grunsky_representation}
Recall the main question posed in the introduction.
 That is, defining $\mathcal{C}_{f}:\mathcal{D}(\Sigma) \rightarrow \mathcal{H}$ as in
  (\ref{eq:C(f)_definition})
 what is
 the image of $\mathcal{C}_{f}$?  The image is a representation of the boundary values of
 the Dirichlet space.

 To answer this recall the operator $\mathcal{W}_f \colon \mathcal{D}^- \rightarrow \mathcal{H}$ defined
 by (\ref{eq:Wf_definition}).
 \[  \mathcal{W}_f = \mathcal{C}_{f} \, \mathbf{I}_f.  \]
 By Theorem \ref{th:I_multi_isomorphism}, the images of $\mathcal{C}_{f}$ and $\mathcal{W}_f$ coincide.
 Next we define the block Grunsky operators as follows.  Let $P_{\pm}(\mathbb{S}^1)$ denote the projections from
 $\mathcal{H}(\mathbb{S}^1)$ to $\mathcal{D}(\disk^\pm)$.
 \begin{definition}[Grunsky operators] \label{de:Grunsky}
  Let $f \in \mathcal{R}$.  There are several cases.
  \begin{enumerate}
  \item For $i,j=0,\ldots,n-1$ define $\gr_{ij}(f) \colon \mathcal{D}(\disk^-) \rightarrow \mathcal{D}(\disk^+)$
  by  \[ \gr_{ij}(f)=P_+(\mathbb{S}^1) \, C_{f_j} \, P(\partial_i \riem)_- \, C_{f^{-1}_i}.  \]
  \item For $i=n$ and $j=0,\ldots,n-1$ define $\gr_{ij}(f) \colon \mathcal{D}(\disk^+) \rightarrow \mathcal{D}(\disk^+)$
  by
  \[  \gr_{ij}(f)= P_+(\mathbb{S}^1) \, C_{f_j} \, P(\partial_n \riem)_+ \, C_{f^{-1}_n}.  \]
  \item For $i=0,\ldots, n-1$ and $j=n$ define $\gr_{ij}(f) \colon \mathcal{D}(\disk^-) \rightarrow \mathcal{D}(\disk^-)$
  by
  \[  \gr_{ij}(f) = P_-(\mathbb{S}^1) \, C_{f_n} \, P(\partial_i \riem)_- \, C_{f^{-1}_i}.  \]
  \item For $i=j=n$ define $\gr_{nn}(f) \colon \mathcal{D}(\disk^+) \rightarrow \mathcal{D}(\disk^-)$ by
  \[  \gr_{nn}(f) = P_-(\mathbb{S}^1) \, C_{f_n} \, P(\partial_i \riem)_+ \, C_{f^{-1}_n}.  \]
  \end{enumerate}
 \begin{remark} \label{re:inotequalj_distinction}  The formulas above require
  a clarification. If $i \neq j$,
  since the image of $f_j$ is compactly contained in $\Omega_i$, the composition
  operator $C_{f_j}$ is a well-defined map from $\mathcal{D}(\Omega_i)$ to $\mathcal{D}(\disk_\pm)$ (where the sign depends on whether $j=n$ or $j \neq n$.)
  Furthermore this operator is clearly bounded by a direct
  computation using a change of variables.

  On the other hand, if $i=  j$, then $f_i$ maps into the complement of $\Omega_i$.
  We compose with an implicit trace $T_i:\mathcal{D}(\Omega_i) \rightarrow \partial \Omega_i$ to the boundary; thus for example if
  $i \neq n$ we ought to write in full
  \begin{equation} \label{eq:full_Grunsky_with_trace}
    \gr_{ii}(f)= P_+(\mathbb{S}^1) \circ C_{f_i} \circ T_i \circ P(\partial_i \riem)_- \circ C_{f_i^{-1}}
  \end{equation}
  with the understanding that $C_{f_i}: \mathcal{H}(\partial_i \riem) \rightarrow \mathcal{H}(\mathbb{S}^1)$.  We define $\gr_{ii}(f)$ similarly when $i=n$
  (replacing $P_+(\mathbb{S}^1)$ with $P_-(\mathbb{S}^1)$).
 \end{remark}

  Finally, we define the operator
  \[  \gr(f): \mathcal{D}^- \rightarrow \mathcal{D}^+  \]
  to be the operator with block structure
  \[  \gr(f) =
  \begin{pmatrix}
  \gr_{00}(f) & \cdots & \gr_{0n}(f) \\
  \vdots & \ddots & \vdots \\
  \gr_{n0}(f) & \cdots & \gr_{nn}(f)
  \end{pmatrix} .
  \]
 \end{definition}

 The diagonal blocks $\gr_{ii}(f)$ are related to the classical Grunsky matrices
 as follows (see  \cite{Durenbook, Pommerenkebook}).   The Grunsky coefficients of $f_n$ are $b_{km}$ where
 \begin{equation} \label{eq:Faber_phin}
  \Phi_n(\Omega_n)(f_k(z))= z^k + k \sum_{m=1}^\infty b_{km} z^{-m} .
 \end{equation}
 The general form of the above follows from Proposition \ref{pr:classical_Grunsky_Faber_proj_identity}.
 It is a classical result that the Grunsky coefficients have the generating function
 \begin{equation} \label{eq:generating_g}
  \log{\frac{f_n(z)-f_n(\zeta)}{z-\zeta}} = \log{f_n'(\infty)} - \sum_{k,m = 1}^\infty b_{km} z^{-k} \zeta^{-m}
 \end{equation}
 (this or a related generating function expression is sometimes given as the definition).
 Therefore we have that if $v_k$ is the function $z \mapsto z^k$ then
 by (\ref{eq:zn_to_Faber})
 \begin{eqnarray} \label{eq:relation_to_regular_Grunsky}
  \gr_{nn}(f) (v_k) & = & P_-(\mathbb{S}^1) \, C_{f_n} \, \mathfrak{I}(\Omega_n,f_n) (v_k) \nonumber \\
  & = & P_-(\mathbb{S}^1) \Phi_n(\Omega_n) \, f_n \nonumber \\
  & = & \sum_{m=1}^\infty k b_{mk} z^{-m}.
 \end{eqnarray}

 Similarly, the matrix of the operator $\gr_{00}(f)$ is the classical Grunsky matrix of $f_0$, although for our
 choices of normalization this has terms involving the coefficients of $\log{f_0(z)/z}$ (this ultimately
 is a result of allowing constants in the Dirichlet space).
 Generally speaking, the Grunsky matrices $\gr_{ii}(f)$ for maps satisfying $f_i(0)=p_i \neq0$ are not common in the literature; however this is just a matter of convention.  Finally, the off-diagonal Grunsky matrices $\gr_{ij}(f)$ are the so-called generalized
 Grunsky matrices originating with Hummel \cite{Hummel}; see also \cite{TT}.
 \begin{remark}
 By comparing equations (\ref{eq:Faber_defn_in}) and (\ref{eq:conjugation_identity_infty}) one can derive an
 explicit formula for the Grunsky coefficients in terms of the coefficients of the Laurent
 series at $\infty$ of $f_n$, due to Jabotinsky \cite[equation (19)]{Jabotinsky} (from whom the derivation of
 \ref{eq:conjugation_identity_infty} is taken; see also \cite{TT}).  A similar formula can be derived for the
 coefficients of $\gr_{00}(f)$; the log terms mentioned above arise because of the relation between
 the $1/z$ coefficient of $(f^{-1}(z))^m$ and those of $\log({f(z)}/{z})$ \cite[Equation (5)]{Jabotinsky}.
 \end{remark}

 We have the following theorem, which follows from our previous work.
 \begin{theorem}  Each block $\gr_{ij}(f)$ is a bounded map.  Thus $\gr(f)$ is bounded.
 \end{theorem}
 \begin{proof} We have that
  \[  \gr_{ij}(f)=P_+(\mathbb{S}^1) \, C_{f_j} \, P(\partial_i \riem)_\pm \, C_{f_i^{-1}}
    = P_{\pm}(\mathbb{S}^1) \, C_{f_j} \, \mathfrak{I}(\Omega_i,f_i). \]
    where the sign of the projection operator is $-$ if $i\neq n$ and $+$ if $i=n$.
  Since $\mathfrak{I}(\Omega_i,f_i)$ is a bounded operator onto $\mathcal{D}(\Omega_i)$ by Corollary \ref{co:I_single_isomorphism}, we need
  only show that $P_+(\mathbb{S}^1) \, C_{f_j}$ is bounded.

  First assume that $i \neq j$.  In that case, as was observed in Remark \ref{re:inotequalj_distinction} the operator $C_{f_j}$ is bounded; the claim thus follows from Theorem \ref{th:Cauchy_bounded}.  Now assume that $i =j$.  In that case
  using equation (\ref{eq:full_Grunsky_with_trace}) to interpret $\gr_{ii}(f)$, the result follows from Theorem \ref{th:Cauchy_bounded}, Theorem \ref{th:trace_of_Dirichlet} and Theorem \ref{th:one-sided_composition}.
 \end{proof}

  If we look at one of the blocks of the Grunsky matrix, we have the interesting
  formula
  \begin{equation} \label{eq:Grunsky_one_block_nice}
  \gr_{nn}(f) = P(\mathbb{S}^1)_- \, C_{f_n} \, P(\partial_n \riem)_+ \, C_{f_n^{-1}}.
  \end{equation}

 \begin{remark}[On defining the Grunsky operator]
  Giving an
  operator definition of the Grunsky matrix is rather involved: one may compose with various isometries, and if one carries
  over the basis isometrically, one has various operators which can justly be called the Grunsky operator.
  These have the same matrix representation but are defined on different function spaces.
  For example, we could have defined our operators on $\mathcal{D}(\Omega_i)$ or $\mathcal{H}(\partial_i \riem)$ rather than on $\mathcal{D}(\disk^\pm)$.

  A more serious complication is that in order to define an operator whose matrix is the Grunsky matrix,
  one must make some assumptions on the mapping function; at a minimum, one must assume univalence.  However, doing so
  rules out one of the most important applications: that the Grunsky inequality is sufficient
  for the existence of a univalent extension to the disk of a locally univalent function.

  Thus when the Grunsky matrix is treated in the literature as an operator,
  restrictions are placed either on the mapping function or on the domain of the operator.
  For example, the domain is often restricted
  to polynomials or a formal completion of polynomials to a Hilbert space
  (e.g.\  \cite{Durenbook, Jabotinsky, Pommerenkebook}).  In that case, the function-theoretic
    interpretation of the domain of the operator tends to be obscured.
The form of the Grunsky matrix given in equation (\ref{eq:Grunsky_one_block_nice})
  is in some sense inherent in the classical Faber polynomial definition, if one restricts
   to polynomials and replaces the Cauchy projection operators with simple truncation.
  In our opinion, our operator formulation (\ref{eq:Grunsky_one_block_nice}) (and the more general Definition \ref{de:Grunsky})
  is of significant interest,
  and seems to require the restriction to WP-class quasicircles.
 \end{remark}
 \begin{remark}[On the Bergman-Schiffer form of the Grunsky operator]
  In defining the Grunsky operator, the weakest possible assumption on the mapping function
  known to the authors is univalence on the disk.  In this case the operator
  can be defined using
  a kernel function of Bergman and Schiffer \cite{BergmanSchiffer, TT} (closely related to Schiffer's formula for the Bergman kernel in terms of Green's
  function).
  However, since in the present paper
  \[ \riem= \sphere \backslash \left( \overline{f_0(\disk^+)} \cup \cdots \cup \overline{f_n(\disk^-)} \right) \]
  must be a Riemann surface (and in particular an open subset of $\sphere$), the assumption of univalence is too weak for our purposes.
 For WP-class quasidisks our formula can be shown to be equivalent to the Bergman-Schiffer form up to composition
  with some isomorphisms.
 \end{remark}
 \begin{remark}  \label{re:Grunsky_on_quasidisks}
  The fact that the Grunsky operator (for example, in Bergman-Schiffer form) and Shen's isomorphism can be defined for quasicircles
  suggests that there might be an extension of our formula (\ref{eq:Grunsky_one_block_nice}) to quasidisks.  As noted in Remark \ref{re:Shen_comparison},
  it is not clear what would take the place of the Besov space $\mathcal{H}(\partial \Omega)$ and the Cauchy
  projection, since the curve $\partial \Omega$ need not be even rectifiable.
 \end{remark}

 Recall the projections $\mathcal{P}_\pm \colon \mathcal{H} \rightarrow \mathcal{D^\pm}$ from Section \ref{se:results}.
 We have the following result.
 \begin{theorem} \label{th:two_projections_id_Grunsky} For any element $f \in \mathcal{R}$ and corresponding mulitply
  connected domain $\Sigma$, we have
 \[ \mathcal{P}_- \, \mathcal{W}_f = \operatorname{Id}  \]
 and
 \[  \mathcal{P}_+ \, \mathcal{W}_f = \gr(f).  \]
 Thus, the image of $\mathcal{W}_f$ is the graph of the Grunsky operator $\gr(f)$.
 \end{theorem}
 \begin{proof} The second identity is just the definition
  of $\gr(f)$.  The first identity follows from (\ref{eq:Faber_phin})  if we can
  show that the $\mathcal{P}_-$ annihilates the off-diagonal matrices.
  However, this is immediate because when $i \neq j$, $C_{f_j}$ maps into
  $\mathcal{D}(\disk^+)$ for $j = 0,\ldots,n-1$ and into $\mathcal{D}(\disk^-)$
  when $j=n$.
 \end{proof}

 Finally, we observe that if $\mathcal{C}_f \colon \mathcal{D}(\partial \riem) \rightarrow \mathcal{H}$ is defined as
 in equation (\ref{eq:composition_shorthand}) (actually, this is a trace followed by a composition operator)
 then we have
 \begin{equation} \label{eq:W_formula}
  \mathcal{W}_f = \mathcal{C}_f \, \mathcal{P}(\riem) \, \mathcal{C}_{f^{-1}} = \mathcal{C}_f \, \mathbf{I}_f
 \end{equation}
 and thus the elegant formula
 \begin{equation}  \label{eq:elegant_formula}
  \gr(f) = \mathcal{P}_+ \, \mathcal{C}_f \, \mathcal{P}(\riem) \, \mathcal{C}_{f^{-1}}.
 \end{equation}
 In particular we have the following result.
 \begin{corollary} \label{co:main_question_answered}  If $f \in \mathcal{R}$, then $\mathcal{C}_f \, \mathcal{D}(\riem)$
 is the graph of $\mathcal{P}_+ \mathcal{W}_f = \gr(f)$.
 \end{corollary}
 \begin{proof}  This follows from Theorem \ref{th:I_multi_isomorphism}, Theorem \ref{th:two_projections_id_Grunsky}
 and equation (\ref{eq:W_formula}).
 \end{proof}

\end{subsection}
\end{section}
\begin{section}{Applications} \label{se:applications}
 The problems solved in this paper are quite natural and can be understood in purely complex function-theoretic terms.
 However, it is motivated by conformal field theory and Teichm\"uller theory.  We now outline this motivation.

  In two-dimensional
 conformal field theory a central object is the class of Riemann surfaces bordered by $n+1$ curves homeomorphic
 to $\mathbb{S}^1$, endowed further with bijective parameterizations $\phi=(\phi_0,\ldots,\phi_n)$ of the boundary by maps
 $\phi_i \colon \mathbb{S}^1 \rightarrow \partial_i \riem$.  Call the set of such Riemann surfaces up to biholomorphisms preserving the
 parameterizations the rigged moduli space.  A heuristically equivalent model of this moduli space is the set of
 Riemann surfaces together with non-overlapping conformal maps of the disk into a compact Riemann surface with distinguished points,
 obtained by sewing on copies of the disk via the parameterizations.
 We are concerned with the genus zero case here.

 One gets differing moduli spaces depending on the analytic
 category of the parameterizations; this corresponds also to the regularity of the
 boundary of the images of the non-overlapping maps. In the conformal field theory literature, the most common choice is
 analytic or $C^\infty$ bijections.  If this condition is weakened to quasisymmetric parameterizations, D. Radnell and E. Schippers showed in  \cite{RadnellSchippers_monster, RS_fiber} that one can draw a correspondence between this moduli space and the quasiconformal
 Teichm\"uller space of bordered Riemann surfaces.  However, this might be too weak a condition
 on the parameterizations in order to carry out some of the constructions in the definition of conformal field theory.  The correct
 class appears to be the WP-class quasisymmetries.  These have sufficient regularity and
 are obtained by completion of the analytic parameterizations (L. Takhtajan and L. P. Teo \cite{TT, TT_expo}). Furthermore,
 WP-class parameterizations maintain the connection
 to the Teichm\"uller theory of bordered surfaces, thanks to the work of G. Cui \cite{Cui}, G. Hui \cite{GuoHui},  Takhtajan and Teo \cite{TT}
 and others. Radnell, Schippers and Staubach have shown in \cite{RSS_Hilbert}  that this connection holds for arbitrary finite genus and number of boundary curves.

 Sewing on disks using WP-class quasisymmetries results in a rigged moduli space consisting of a Riemann surface
 with conformal maps onto non-overlapping WP-class quasidisks.  In the case of one boundary curve, this results
 in a model of the universal Teichm\"uller space which is a Hilbert
 manifold rather than a Banach manifold  \cite{GuoHui, TT}.

 The construction of genus-zero conformal field theory from vertex operator algebras was accomplished by Y.-Z. Huang and L. Kong in \cite{HuangKong}. It is based on the deep mathematical notion of a chiral conformal field theory which was constructed by Huang (see \cite{Huang_review03,Huang_DiffEq}). The foundation of these results is the isomorphism, established by Huang in \cite{Huang}, between the category of vertex operator algebras and the category of geometric vertex operator algebras based on the rigged moduli.  A fundamental part of the chiral theory is the construction of the determinant line bundle of $\overline{\partial} \oplus \text{pr}$ over the rigged moduli space, where $\text{pr}$ is the projection of the boundary values onto
 the space of those boundary values with Fourier series with only
 negative terms ($a_n e^{-in\theta}$, $n>0$) under the parametrizations.  Equivalently
 one can construct the determinant line bundle of the map $\pi$ which is the projection
 from holomorphic functions on the Riemann surface to the set of negative Fourier series.
  \cite[Proposition D.3.3]{Huang}. In order
 to define the determinant line it is necessary to use the Sokhotski-Plemelj jump formula and the space of
 holomorphic functions $\mathcal{C}_f \, \mathcal{D}(\riem)$ defined in Section \ref{se:results}. In \cite{Huang}, the parameterizations
 are assumed to be analytic, and the holomorphic functions on the Riemann surface $\riem$ were
   assumed to be smooth up to the boundary.  Thus he initially avoided these analytic difficulties.  However, in order to
 apply the theory of elliptic operators an analytic completion was necessary, which resulted in
 the class of Sobolev functions $H^{s-\frac{1}{2}}(\partial \riem)$ \cite[Appendix D.3]{Huang}, and the analytic requirements
 were satisfied for $s \geq 1$.

 For analytic curves, this $H^{\frac{1}{2}}(\partial \Sigma)$ space is clearly the boundary values of holomorphic functions of bounded
 Dirichlet energy. As mentioned above,
 Takhtajan and Teo \cite{TT, TT_expo} showed that the WP-class quasisymmetries are the analytic completion of the analytic
 bijections of the circle, and furthermore that they comprise a topological group.  These
 results strongly
 motivate the extension of Huang's construction to Riemann surfaces
 with borders parameterized by WP-class
 quasisymmetries.  Thus, we are required to extend the
 Sokhotski-Plemelj jump problem to WP-quasicircles, and also find the Hilbert space of boundary values
 of holomorphic/anti-holomorphic functions of finite Dirichlet energy.
 In this paper, we have shown that for WP-class quasicircles, the role of  $H^{\frac{1}{2}}(\partial \riem)$ for more
 regular curves is taken on
 by a the Besov space $\mathcal{H}(\Gamma)$, and that the Cauchy integral
 operator in the Sokhotski-Plemelj jump formula is indeed a bounded projection for these curves and this class
 of boundary values. This is a necessary step in the construction of the determinant line bundle.

 Further evidence that the WP-class parameterizations are an appropriate choice for
 the parameterizations comes from the following striking
 results of
 Takhtajan and Teo \cite{TT}.  First, in the case of one boundary curve in genus zero, the Grunsky matrix provides an embedding
 of the universal Teichm\"uller space into an infinite-dimensional Siegel upper half plane (that is, the Segal-Wilson
 universal Grassmanian).
 Second, the Grunsky operators are Hilbert-Schmidt precisely for WP-class quasicircles (see also Hui \cite{GuoHui} and Y. Shen \cite{ShenFaber, Shen}). In the setting of Huang \cite{Huang},
 this fact ought to translate into sufficient regularity for the existence of a determinant of
 $\pi$, given the relation between the Grunsky operator and the Sokhotski-Plemelj jump formula
 obtained
 here (equation (\ref{eq:elegant_formula})), and the relation between this decomposition and $\pi$ obtained already
   in \cite{Huang}.  A crucial point is that our formula for the Grunsky operator depends
   explicitly on the Cauchy projection
   associated with the quasicircle.

 Aside from the problems described above, there is increasing interest in the WP-class universal Teichm\"uller
 space.  An overview of the literature can be found in the introduction of the recent paper of Shen \cite{Shen_characterization}.
 That paper also contains the solution to the problem of intrinsically characterizing WP-class quasisymmetries.
\end{section}

\end{document}